\documentclass[12pt]{article}

\usepackage{amsmath,amsbsy}
\usepackage{microtype}
\usepackage{amsthm,amsfonts,yfonts}
\usepackage{amssymb}
\usepackage{mathrsfs,microtype}
\usepackage[textwidth=6in]{geometry}
\usepackage[usenames]{color}
\usepackage{graphicx}
\usepackage[table]{xcolor}
\usepackage{tikz}
\usepackage{verbatim}
\usepackage{mathdots}
\usepackage{mathtools}
\usepackage{multirow}
\usepackage{textcomp}
\usepackage{bbm}
\usetikzlibrary{arrows,decorations.pathmorphing,decorations.footprints,fadings,calc,trees,mindmap,shadows,decorations.text,patterns,positioning,shapes,matrix,fit,through,decorations.pathreplacing}
\usepackage{hyperref}

\theoremstyle{plain}
\newtheorem{theorem}{Theorem}[section]

\newtheorem{corollary}[theorem]{Corollary}

\newtheorem*{theorem-BGcompCrit}{Theorem \ref{thm-BGcompCrit}}
\newtheorem*{theorem-E[a^kb^l]}{Theorem \ref{thm-E[a^kb^l]}}
\newtheorem*{theorem-betaDist}{Theorem \ref{thm-betaDist}}
\newtheorem*{theorem-bgPermComp}{Theorem \ref{thm-bgPermComp}}
\newtheorem*{theorem-3chains}{Theorem \ref{thm-3chains}}

\theoremstyle{definition}
\newtheorem{definition}[theorem]{Definition}
\newtheorem{example}[theorem]{Example}
\newtheorem*{notation}{Notation}

\theoremstyle{remark}
\newtheorem*{remark}{Remark}

\newcommand{\frakS}{\mathfrak{S}}

\newcommand{\pr}{\mathbb{P}}
\newcommand{\E}{\mathbb{E}}
\newcommand{\Var}{\mathbb{V}}

\newcommand{\inv}{\textrm{inv}}

\newcommand{\bG}{\mathfrak{B}}
\newcommand{\bg}{\mathfrak{b}}
\newcommand{\bl}{\boldsymbol{\ell}}
\newcommand{\w}{\omega}

\def\epf{\quad \qed} 
\def\({\left(} 
\def\){\right)}
\def\l[{\left[}
\def\r]{\right]}
\def\vtl{\vartriangleleft}

\definecolor{light-gray}{gray}{0.75}

\title{On comparability of bigrassmannian permutations}
\author{John Engbers\thanks{john.engbers@marquette.edu; Department of Mathematics, Statistics and Computer Science, Marquette University,
Milwaukee, WI 53201, USA. Research supported by the Simons Foundation grant 524418.}
\and
Adam Hammett\thanks{ahammett@cedarville.edu; Department of Science and Mathematics, Cedarville University, Cedarville, OH 45314, USA.}} 
\date{\today}

\begin{document}

\maketitle

\begin{abstract}
Let $\mathfrak{S}_n$ and $\mathfrak{B}_n$ denote the respective sets of ordinary and bigrassmannian (BG) permutations of order $n$, and let $\(\mathfrak{S}_n,\le\)$ denote the Bruhat ordering permutation poset. We study the restricted poset $\(\mathfrak{B}_n,\le\)$, first providing a simple criterion for comparability. This criterion is used to show that that the poset is connected, to enumerate the saturated chains between elements, and to enumerate the number of maximal elements below $r$ fixed elements. It also quickly produces formulas for $\beta(\w)$ ($\alpha(\w)$ respectively), the number of BG permutations weakly below (weakly above respectively) a fixed $\w\in\bG_n$, and is used to compute the M\"obius function on any interval in $\bG_n$.

We then turn to a probabilistic study of $\beta=\beta(\w)$ ($\alpha=\alpha(\w)$ respectively) for the uniformly random $\w\in\bG_n$. 
We show that $\alpha$ and $\beta$ are equidistributed, and that $\beta$ is of the same order as its expectation 
with high probability, but fails to concentrate about its mean. This latter fact derives from the limiting distribution of $\beta/n^3$. 

We also compute the probability that randomly chosen BG permutations form a 2- or 3-element multichain.

\end{abstract}

\section{Introduction}\label{sec-intro}

Let $n\ge 1$ be an integer, and let $[n]:=\{1,2,\ldots,n\}$. Bigrassmannian elements of a Coxeter group are elements that have exactly one left descent and exactly one right descent \cite{GeckKim}. In this paper, we focus on the symmetric group of order $n$ permutations $\frakS_n$, which is a Coxeter group of type $A_{n-1}$. A number of recent papers have studied bigrassmannian elements in $\frakS_n$ and other Coxeter groups, see for example \cite{BrSc,GeckKim,Kobayashi2,Kobayashi,Kobayashi3,Kobayashi4,LascouxSchutz,RWY,WY}. We write an element $\w \in \frakS_n$ in one-line array notation $\w=\w(1)\w(2)\cdots\w(n)$, so that $\w(i)$ is the image of $i$ under $\w$. Here the bigrassmannian (BG) permutations are those elements $\w \in \frakS_n$ such that $\w$ and $\w^{-1}$ admit a unique \emph{descent}, i.e. a unique location-pair $(i,j)\in [n-1]^2$ such that $\w(i)>\w(i+1)$ and $\w^{-1}(j)>\w^{-1}(j+1)$. Let $\bG_n$ denote the set of BG permutations in $\frakS_n$. Then $\w \in \bG_n$ if and only if there is a triple $0 \leq a < b < c \leq n$ such that 
\begin{equation}\label{eq-BGform}
\w = 1\cdots a(b+1)\cdots c(a+1)\cdots b(c+1)\cdots n
\end{equation}
in one-line array notation (see, e.g., \cite[p. 169, ex. 39]{BjornerBrenti1}). Note that $a=0$ and $c=n$ are permitted here. It is clear from (\ref{eq-BGform}) that $\bg_n:=|\bG_n|=\binom{n+1}{3}$.

Recall that the symmetric group $\frakS_n$ equipped with the \emph{Bruhat order} ``$\leq$'' becomes a poset (see, e.g., \cite{BjornerBrenti1}). Here is a precise definition of Bruhat order on the set of permutations $\frakS_n$ (see Stanley \cite[p. 172, ex. 24]{Stanley1}). If $\w\in \frakS_n$ then a \textit{reduction} of $\w$ is a permutation obtained from $\w$ by interchanging some $\w(i)$ with some $\w(j)$ provided $i<j$ and $\w(i)>\w(j)$; in other words, the location-pair $(i,j)$ forms an \emph{inversion} of $\w$. We say that $\pi \leq \sigma$ in the Bruhat order if there is a chain $\sigma=\w_1\to\w_2\to\cdots \to \w_s= \pi$, where each $\w_t$ is a reduction of $\w_{t-1}$. The number of inversions in $\w_t$ strictly decreases with $t$. Indeed, one can show that if $\w_2$ is a reduction of $\w_1$ via the interchange $\w_1(i)\leftrightarrow \w_1(j)$, $i<j$, then
\begin{eqnarray*}
& \inv(\w_1)=\inv(\w_2)+2N(\w_1)+1,\\
&N(\w_1):=|\{ k \, : \, i<k<j,\, \w_1(i)>\w_1(k)>\w_1(j)\}|;
\end{eqnarray*}

\noindent here $\inv(\bullet)$ is the number of inversions in $\bullet$. The Bruhat order notion can be extended to other Coxeter groups \cite{Bjorner}, and bigrassmannian elements have been used to investigate the structure of the Bruhat order \cite{BjornerBrenti2,GeckKim,LascouxSchutz}. 

A large portion of this paper is devoted to studying comparable BG permutations, where comparability is inherited from $\(\frakS_n,\le\)$. Figure \ref{BGPoset} illustrates the poset of BG permutations for $n=3$ and $n=4$.

\begin{figure}[tbh!]
\begin{center}
\includegraphics[trim = 1.55in 8.45in 1.55in 1.15in, clip, scale = .65]{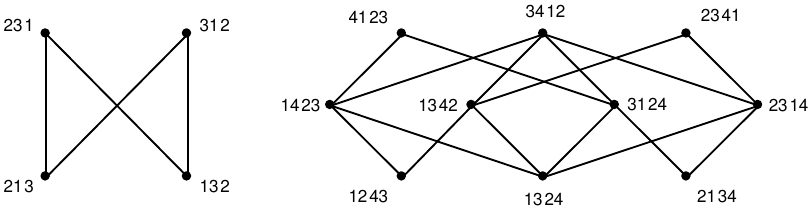}
\end{center}
\caption{The posets $(\bG_3,\le)$ and $(\bG_4,\le)$.}
\label{BGPoset}
\end{figure}

A sequence of reductions, when starting from an element $\sigma\in\bG_n$ (viewing $\sigma \in \frakS_n$), will not necessarily keep us within the collection of BG permutations. However,  there are efficient algorithms for checking Bruhat comparability that do not rely upon this reduction operation. The Ehresmann tableau criterion \cite{Ehresmann} states that $\pi \le \sigma $ if and only if $\pi_{i,j}\le \sigma_{i,j}$ for all $1\le i\le j\le n$, where $\pi_{i,j}$ and $\sigma_{i,j}$ are the $i$th entries in the increasing rearrangement of $\pi(1),\dots,\pi(j) $ and of $\sigma(1),\dots,\sigma(j)$. For example, comparability of the BG permutations $14235<34512$ is verified by entry-wise comparisons of the two tableaux

\[
\begin{array}{c}
\begin{array}{ccccc}
1 & & & & \\ 1 & 4 & & & \\ 1 & 2 & 4 & & \\ 1 & 2 & 3 & 4 & \\ 1 & 2 & 3 & 4 & 5 
\end{array}
\qquad \qquad
\begin{array}{ccccc}
3 & & & & \\ 3 & 4 & & & \\ 3 & 4 & 5 & & \\ 1 & 3 & 4 & 5 & \\ 1 & 2 & 3 & 4 & 5 
\end{array}
\end{array}.
\]

\noindent These tableaux represent \emph{monotone triangles} (or \emph{Gog triangles}, in the terminology of Zeilberger \cite{Zeilberger}) formed from the two permutations. Monotone triangles are well-known to be in bijection with the collection of alternating sign matrices \cite{Bressoud}, which have been of ubiquitous combinatorial interest in recent years, and the size $n$ monotone triangles under entry-wise comparisons constitute the unique MacNeille completion of $(\frakS_n,\le)$ to a lattice \cite[ex. 7.103]{Stanley2}. A. Lascoux and M.~P. Sch\"utzenberger \cite{LascouxSchutz} showed that $\bG_n$ is precisely the set of join-irreducible elements of $(\frakS_n,\le)$, and so by Birkhoff's representation theorem \cite{Birkhoff} the lattice of lower sets of $(\bG_n,\le)$ under containment is order-isomorphic to the lattice of monotone triangles. The lattice of monotone triangles is quite important, but notoriously difficult to say much about. Recently in \cite{EngbersHammett}, the present authors proved that given a set of $r$ independent and uniformly random monotone triangles of size $n$, the probability that the largest element dominated by all of them equals the minimum element of the lattice is asymptotically $r/M_n$ as $n\to\infty$, where $M_n$ is the number of monotone triangles of size $n$. Therefore, it is our hope that understanding $\(\bG_n,\le\)$ can aid in our understanding of the lattice of monotone triangles; this is the perspective taken in \cite{BrSc}.

The Ehresmann tableau criterion requires that $\Theta(n^2)$ conditions be checked. Bj\"orner and Brenti \cite{BjornerBrenti2} discovered an improved tableau criterion (based upon Deodhar's more general Coxeter group characterization in \cite{Deodhar}) that requires far fewer comparisons. For BG permutations, this criterion requires only $O(n)$ comparisons. Indeed, given $\pi,\sigma\in\bG_n$, to determine whether $\pi\le\sigma$ we need only check the row of the two tableaux that corresponds to the unique descent of $\pi$. So in our example above, $14235<34512$ is verified more efficiently by entry-wise comparisons in the singular row

\[
\begin{array}{c}
\begin{array}{cc}
1 & 4  
\end{array}
\qquad \qquad
\begin{array}{cc}
3 & 4   
\end{array}
\end{array}.
\]

Our point of departure for this paper is the following new simple characterization of comparability for BG permutations, which we prove in Section \ref{sec-compcrit}. An alternate characterization of comparability, involving ordered triples, is given by Nathan Reading in \cite{Reading}. To state our characterization, we first define a vector that encapsulates the information contained in (\ref{eq-BGform}).

\begin{definition}\label{def-BGvector}
Let $\sigma \in \bG_n$ with $a$, $b$, and $c$ as in (\ref{eq-BGform}).  We define the \emph{lengths} of $\sigma$ to be $\ell_1(\sigma) := a+1$, $\ell_2(\sigma) := b-a$, $\ell_3(\sigma) := c-b$, and $\ell_4(\sigma) := n+1-c$. The \emph{length vector} of $\sigma$ is $\bl(\sigma):=(\ell_1(\sigma),\ell_2(\sigma),\ell_3(\sigma),\ell_4(\sigma))$. We shall occasionally suppress the argument $\sigma$ and write only $\bl=(\ell_1,\ell_2,\ell_3,\ell_4)$ for the length vector.
\end{definition}

\noindent A couple of notes are in order. First, 
\begin{enumerate}
\item[$(\dagger)$] $\ell_1+\ell_2+\ell_3+\ell_4=n+2,\quad \ell_i\in [n-1]$.
\end{enumerate}
Furthermore any choice of $(\ell_1,\ell_2,\ell_3,\ell_4)$ satisfying $(\dagger)$ corresponds to a unique element of $\bG_n$. 

\begin{theorem}\label{thm-BGcompCrit}
Let $\pi,\sigma\in\bG_n$. The following are equivalent:
\begin{enumerate}
\item[$(1)$] $\pi \leq \sigma$  and
\item[$(2)$] $\ell_1(\pi)\geq \ell_1(\sigma), \ell_4(\pi) \geq \ell_4(\sigma), \ell_2(\pi) \leq \ell_2(\sigma), \ell_3(\pi) \leq \ell_3(\sigma)$.
\end{enumerate}
\end{theorem}

Theorem \ref{thm-BGcompCrit} allows for a quick count of the number of BG permutations weakly above and below $\sigma\in\bG_n$. We remark that the count below $\sigma$ was studied by Kobayashi \cite{Kobayashi} in the setting of $\sigma \in \frakS_n$.  An equivalent result, discovered independently, occurs as Theorem 31 in \cite{BrSc}.

\begin{notation}
Given $\sigma \in \frakS_n$, let $\beta(\sigma)$ denote the number of $\pi\in\bG_n$ such that $\pi \leq \sigma$ in Bruhat order, and let $\alpha(\sigma)$ denote the number of $\tau\in \bG_n$ such that $\sigma \leq \tau$. 
\end{notation}

For $\sigma \in \bG_n$, in the notation of (\ref{eq-BGform}) we show in Corollary \ref{cor-BGcompCrit} that
\[
\beta(\sigma) = \frac{1}{2}(b-a)(c-b)(c-a) \qquad \text{and} \qquad \alpha(\sigma) = \frac{1}{2}(a+1)(n-c+1)(n-c+1+a+1). 
\]

\noindent Theorem \ref{thm-BGcompCrit} also reveals that $\(\bG_n,\le\)$ is \emph{self-dual}, i.e. $\(\bG_n,\le\)$ is order isomorphic to $\(\bG_n,\ge\)$. We prove this self-duality in Corollary \ref{cor-BGcompCrit}, and discuss how the Hasse diagram of $\(\bG_n,\le\)$ differs from viewing bigrassmannian elements as elements in $\(\frakS_n,\le\)$.

After this, we briefly examine $(\bG_n,\le)$ from a structural perspective in Section \ref{sec-structure}. Previously, Reading \cite{Reading} showed that the poset is ranked and discussed the number of elements in each rank-level and the number of covering relations.  We enumerate the number of chains between two comparable elements and the number of maximal chains (Theorem \ref{thm-number of r chains}). In Theorem \ref{thm-conn} we compute the distance between any two elements, which implies that the poset is connected. The results of Theorems \ref{thm-butterfly} and \ref{thm-inf} describe specific properties of $\(\bG_n,\le\)$ related to how far removed this poset is from being a lattice, and include the enumeration of the number of maximal elements below $r$ fixed elements of the poset (Theorem \ref{thm-inf}). We also compute the M\"obius function on any interval in $\bG_n$, showing that it takes values in the set $\{-2,-1,0,1,3,4\}$ and that all of these values occur for some interval in $\bG_n$ for all $n \geq 6$ (Theorem \ref{thm-Mobius}).

Then we turn to a probabilistic study of $(\bG_n,\le)$. In Section \ref{sec-probability} we analyze a uniformly random element of $\bG_n$, starting with calculation of the product moments $\E\left[ \alpha^k \beta^\ell\right]$. To state this result, let $g_k(t) := \sum_{i \geq 0} i^k t^i$ for each $k \geq 0$ (where $0^0:=1$), $\l[x^j\r]G(x)$ denote the coefficient of $x^j$ in a generating function $G(x)$, and ${\bf 1}_{\{k=0\}}$ denote the indicator of the event $\{k=0\}$.

\begin{theorem}\label{thm-E[a^kb^l]}
For each $k\ge 0$, $\ell\ge 0$, and $n\ge 2$ we have 
\[
\E\l[\alpha^k\beta^\ell\r]=\frac{k!\ell!}{2^{k+\ell}\bg_n}\l[t^{n+2}u^kv^\ell\r]\frac{\(g_k\(te^{(1-t)u}\)-{\bf 1}_{\{k=0\}}\)^2\(g_\ell\(te^{(1-t)v}\)-{\bf 1}_{\{\ell=0\}}\)^2}{(1-t)^{k+\ell}}.
\]
\end{theorem}

\noindent Implicit in Theorem \ref{thm-E[a^kb^l]} is the equidistribution of $\alpha$ and $\beta$, the mean and variance of $\beta$ (Theorem \ref{thm-bgE[beta]}), and the asymptotic moments of all order for these random variables (Theorem \ref{thm-E[a^kb^l]sim}). We also calculate the limiting distribution of $\beta$ ($\alpha$ respectively), properly scaled. Letting $L_1=X$, $L_2=Y-X$, $L_3=Z-Y$ and $L_4=1-Z$ be the random variables arising from rank-ordering three independent and uniformly random points from the unit interval $0<X<Y<Z<1$, we have the following.
 
\begin{theorem}\label{thm-betaDist}
The random variable $\beta/n^3$ converges in distribution to $\frac{1}{2}L_2L_3(L_2+L_3)$. More precisely, for intervals $I=(ni_0,ni_1)\subseteq (0,n)$ and $J=(nj_0,nj_1)\subseteq (0,n)$ we have 
\[
\pr\(\ell_2\in I, \ell_3\in J\)
\to \int_{\substack{i_0<x_1<i_1 \\ j_0<x_2<j_1\\ 0< x_1+x_2 < 1}} 6(1-x_1-x_2) \,dA=\pr\l[L_2\in (i_0,i_1),L_3\in(j_0,j_1)\r],
\]
$n\to\infty$, with error term of order $n^{-1}$.
\end{theorem}
\noindent From this it follows that $\beta$ is not concentrated about its mean, and we also find the scaled limiting moments of $\beta$ (Theorem \ref{thm-betaMomentsIntegral}).

Finally, in Section \ref{sec-comparability} we show how certain precise moment calculations implicit in Theorem \ref{thm-E[a^kb^l]} produce the exact probabilities of comparability among pairs and triples of randomly chosen elements of $\bG_n$.  In particular, we show the following. 

\begin{theorem}\label{thm-bgPermComp}
Let $\pi,\sigma\in\bG_n$ be selected independently and uniformly at random, and let 
\[
p_{n,2}:=\pr\( \pi \textrm{ and } \sigma \textrm{ are comparable} \) \qquad \textrm{and}\qquad  p_{n,2,\leq}:= \pr\( \pi \le \sigma \). 
\]
Then for $n\ge 2$,
\[
p_{n,2}=\frac{\bg_{n+2}-5}{5\bg_n} \qquad \textrm{and}\qquad  p_{n,2,\leq}=\frac{\bg_{n+2}}{10\bg_n}.
\]
\end{theorem}

\begin{theorem}\label{thm-3chains}
Let $\pi, \sigma, \tau\in \bG_n$ be selected independently and uniformly at random, and let
\[
p_{n,3}:=\pr\( \pi,\sigma,\tau \textrm{ are comparable} \) \qquad \textrm{and}\qquad  p_{n,3,\leq}:= \pr\( \pi \le \sigma \le \tau \). 
\]
Then for $n\ge 2$,
\[
p_{n,3}= \frac{(n^2+4n+6)\bg_{n+3}\bg_{n+6} - 42(n+6)(n+7)\bg_{n+2} + 420(n+6)(n+7)}{70(n+6)(n+7)\bg_n^2}
\]
and
\[
p_{n,3,\leq}=\frac{(n^2+4 n+6)\bg_{n+3}\bg_{n+6}}{420(n+6)(n+7)\bg_n^2}.
\]
\end{theorem}

\noindent In fact, we derive both of these as special cases of the result for $(k,\ell)$-stars (Theorem \ref{thm-q_n(k,l)}); we leave the details of this generalization to Section \ref{sec-comparability}. 
From Theorems \ref{thm-bgPermComp} and \ref{thm-3chains} we deduce that the proportion of pairs $(\pi,\sigma)$ with $\pi\le\sigma$ decreases to $\frac{1}{10}$ as $n \to \infty$ (Corollary \ref{cor-bgPermComp}), and the proportion of triples $(\pi,\sigma,\tau)$ with $\pi\le\sigma\le\tau$ decreases to $\frac{1}{420}$ as $n \to \infty$ (Corollary \ref{cor-3chains}).  These results stand in stark contrast to the analogous ``probability-of-comparability'' questions addressed in \cite{Hammett,HammettPittel}, where for uniformly random and independent $\pi,\sigma,\tau\in\frakS_n$ the respective probabilities that $\pi\le\sigma$ and $\pi\le\sigma\le\tau$ were shown to tend to $0$ as $n\to\infty$. 

Henceforth, we shall use both $\(\bG_n,\le\)$ and $\bG_n$ to refer to the poset of BG permutations.

\section{Proof of Theorem \ref{thm-BGcompCrit}}\label{sec-compcrit}

The goal of this section is to prove Theorem \ref{thm-BGcompCrit} and then provide some quick consequences. Before proceeding we mention that bigrassmannian permutations fall into the class of \emph{fully commutative elements}, and Stembridge's papers \cite{Stembridge3,Stembridge2,Stembridge1} provide a bird's eye view of such elements. In particular, his development of the ``heap'' poset there provides an alternative proof of our Theorem \ref{thm-BGcompCrit}. We give a short elementary proof here that relies only on the length vector concept, and fits with the spirit of the comparability results mentioned in Section \ref{sec-intro}. 

To begin, recall that for $\sigma \in \bG_n$ with $a$, $b$, and $c$ as in (\ref{eq-BGform}), the length vector of $\sigma$ is
$\bl(\sigma):=(\ell_1,\ell_2,\ell_3,\ell_4)$ where $\ell_1= a+1$, $\ell_2 = b-a$, $\ell_3 = c-b$, and $\ell_4 = n+1-c$.  Moreover, we have 
\begin{enumerate}
\item[$(\dagger)$] $\ell_1+\ell_2+\ell_3+\ell_4=n+2,\quad \ell_i\in [n-1]$.
\end{enumerate}
And conversely any choice of $(\ell_1,\ell_2,\ell_3,\ell_4)$ satisfying $(\dagger)$ corresponds to a unique element of $\bG_n$. 

\begin{definition}
Let $\sigma\in \bG_n$ have length vector $(\ell_1,\ell_2,\ell_3,\ell_4)$. Define the map $f_{2143}:\bG_n \to \bG_n$ so that $f_{2143}(\sigma)$ is the element of $\bG_n$ with length vector $(\ell_2,\ell_1,\ell_4,\ell_3)$.
\end{definition}

Note that $(f_{2143})^2$ is the identity map, which implies that $f_{2143}$ is a bijection (involution). 
The similarly defined map $f_{1324}$ corresponds to the inverse map on $\bG_n$, and $f_{4321}$ corresponds to the conjugate map $\bar{\sigma}$ (which reverses both the permutation and the rank, in other words, is defined by $\bar{\sigma}(i) = n+1-\sigma(n+1-i)$). In fact, we can define the bijection $f_\phi$ on $\bG_n$ for any $\phi \in \frakS_4$.

\begin{example}
Consider $\sigma = 15234678 \in \bG_8$, 
which has length vector $(2,3,1,4)$.  Then $f_{2143}(\sigma)$ has length vector $(3,2,4,1)$, and so $f_{2143}(\sigma) = 12567834$.
\end{example}

Here is now the key theorem, which we restate from Section \ref{sec-intro}, that reframes comparability entirely in terms of the coordinates of the length vector.  

\begin{theorem-BGcompCrit}
Let $\pi,\sigma\in\bG_n$. The following are equivalent:
\begin{enumerate}
\item[$(1)$] $\pi \leq \sigma$  and
\item[$(2)$] $\ell_1(\pi)\geq \ell_1(\sigma), \ell_4(\pi) \geq \ell_4(\sigma), \ell_2(\pi) \leq \ell_2(\sigma), \ell_3(\pi) \leq \ell_3(\sigma)$.
\end{enumerate}
\end{theorem-BGcompCrit}

\begin{proof}
Throughout the proof, for brevity we write $\ell_i:=\ell_i(\pi)$ and $m_i:=\ell_i(\sigma)$, $i\in [4]$. Notice that $\pi(i)=i$ for $1 \leq i \leq \ell_1-1$, $\pi(i) = i+\ell_2$ for $\ell_1 \leq i \leq \ell_1+\ell_3-1$, $\pi(i) = i-\ell_3$ for $\ell_1+\ell_3 \leq i \leq \ell_1+\ell_2+\ell_3-1$, and $\pi(i) = i$ for $\ell_1+\ell_2+\ell_3 \leq i \leq n$.  Similarly, we have $\sigma(i)=i$ for $1 \leq i \leq m_1-1$, $\sigma(i) = i+m_2$ for $m_1 \leq i \leq m_1+ m_3-1$, $\sigma(i) = i- m_3$ for $m_1+ m_3 \leq i \leq m_1+m_2+ m_3-1$, and $\sigma(i) = i$ for $m_1+m_2+ m_3 \leq i \leq n$.

Suppose first that $\pi \leq \sigma$. Then by the Ehresmann tableau criterion the first $i$ elements of $\pi$ (in increasing order) are at most the first $i$ elements of $\sigma$ for each $i$. Suppose that $\ell_1<m_1$. Then $\pi(\ell_1) > \ell_1 = \sigma(\ell_1)$, which combined with $\sigma(i) = \pi(i) = i$ for $1 \leq i < \ell_1$ creates a contradiction.  Therefore $\ell_1\ge m_1$. Next, suppose that $\ell_2>m_2$. Then $\pi(\ell_1) = \ell_1+\ell_2$, and $\sigma(i) \leq i+m_2<i+\ell_2$ for all $1 \leq i \leq \ell_1$.  Therefore by focusing on the first $\ell_1$ terms of $\pi$ and $\sigma$ in increasing order we see that $\pi \nleq \sigma$. This contradiction implies that $\ell_2\le m_2$. Now, recall that $\bar{\pi}$ reverses the permutation and the rank of $\pi$.  Then it is easy to see, using the Ehresmann tableau criterion, that $\pi \leq \sigma$ if and only if $\bar{\pi} \leq \bar{\sigma}$.  Briefly, this is because reversing the rank of the permutations reverses the original directions of the entry-wise inequalities, and then reversing the rank-reversed permutations simply puts these inequalities back to their original directions. Combining this observation with the first two inequalities gives the remaining two inequalities.

\medskip

Suppose next that $\ell_1 \geq  m_1$, $\ell_4 \geq  m_4$, $\ell_2 \leq  m_2$, and $\ell_3 \leq  m_3$. We want to show that $\pi \leq \sigma$.  By the Bj\"{o}rner-Brenti criterion, we only need to show that at position $\ell_1+\ell_3-1$ (the position of the unique descent of $\pi$) the entries of $\pi$ in increasing order are at most the entries of $\sigma$ in increasing order. The entries of $\pi$ are $1,2,\ldots,\ell_1-1,\ell_1+\ell_2,\ell_1+\ell_2+1,\ldots,\ell_1+\ell_2+\ell_3-1$.  What are the first $\ell_1+\ell_3-1$ entries, in order, of $\sigma$? Since $\ell_4 \geq m_4$, we know that $\ell_1+\ell_2+\ell_3 \leq  m_1+ m_2+ m_3$. This means that the entries for $\sigma$ form two intervals, the lower starting at $1$ and the higher starting at $ m_1+ m_2$.  If $\ell_1+\ell_3 <  m_1+ m_3$, meaning the position of the descent for $\sigma$ is larger than the position of the descent for $\pi$, then the result follows from $\ell_1 \geq  m_1$ and $\ell_2 \leq  m_2$ (as the entries from $\sigma$ are $1$, $\ldots$, $m_1-1$, $m_1+m_2$, $\ldots$, $\ell_1+\ell_3-1+m_2$). If $\ell_1+\ell_3 \geq m_1 + m_3$, then since $\ell_1+\ell_2+\ell_3 \leq m_1 + m_2 + m_3$ the higher interval for $\sigma$ goes from $m_1+m_2$ to $m_1 + m_2 + m_3-1$ and the lower interval goes from $1$ to $\ell_1+\ell_3-m_3-1$. Since $m_3 \geq \ell_3$, again the result follows.  These are all possible cases, and so the theorem is proved.
\end{proof}

To summarize, comparability within $\bG_n$ is determined entirely by analyzing length vectors. The following is an important consequence of Theorem \ref{thm-BGcompCrit}. 

\begin{corollary}\label{cor-BGcompCrit}
Let $\pi, \sigma \in \bG_n$.  Then 
\begin{enumerate}
\item[$(1)$] $\pi \leq \sigma$ if and only if $f_{2143}(\pi) \geq f_{2143}(\sigma)$,
\item[$(2)$] $\beta(\sigma) = \frac{1}{2}\ell_2 \ell_3(\ell_2+\ell_3)$,
\item[$(3)$] $\alpha(\sigma) = \frac{1}{2}\ell_1 \ell_4(\ell_1+\ell_4)$,
\item[$(4)$] $\beta(\sigma) = \alpha(f_{2143}(\sigma))$.
\end{enumerate}
\end{corollary}

\begin{proof}
(1) is clear from Theorem \ref{thm-BGcompCrit}. We will prove (2) only, as (3) follows from a similar argument, and the combination of (2) and (3) imply (4).

Suppose that $\sigma$ has length vector $(\ell_1,\ell_2,\ell_3,\ell_4)$.  To find $\beta(\sigma)$, we must find all elements of $\bG_n$ with length vectors $(m_1,m_2,m_3,m_4)$ so that $\ell_1 \leq m_1$, $\ell_2 \geq m_2$, $\ell_3 \geq m_3$, and $\ell_4 \leq m_4$. We imagine removing $i$ from $\ell_2$ and $j$ from $\ell_3$ and putting these on the outer two lengths; there are $i+j+1$ ways that this can be done.  Therefore
\begin{equation*}
\beta(\sigma) = \sum_{\substack{0\leq i \leq \ell_2-1\\ 0 \leq j \leq \ell_3-1}} (i+j+1) = \ell_2\binom{\ell_3}{2} + \ell_3\binom{\ell_2}{2} + \ell_2\ell_3 = \frac{1}{2}\ell_2\ell_3(\ell_2+\ell_3). \qedhere
\end{equation*} 
\end{proof}

As we mentioned in Section \ref{sec-intro}, items (1) and (4) of Corollary \ref{cor-BGcompCrit} illuminate the self-duality of the BG poset. 
See Figure \ref{B5Poset}, where we have highlighted the \emph{down-set} of $41235$ (i.e., the set of all $\pi\in\bG_5$ with $\pi\le 41235$) as well as the \emph{up-set} of $12453$ to illustrate 
\[
\beta(41235)=\alpha\(f_{2143}(41235)\)=\alpha(12453)=6.
\]

\begin{figure}[tbh!]
\begin{center}
\includegraphics[trim = 1.15in 6.70in 1.95in 1.15in, clip, scale = .65]{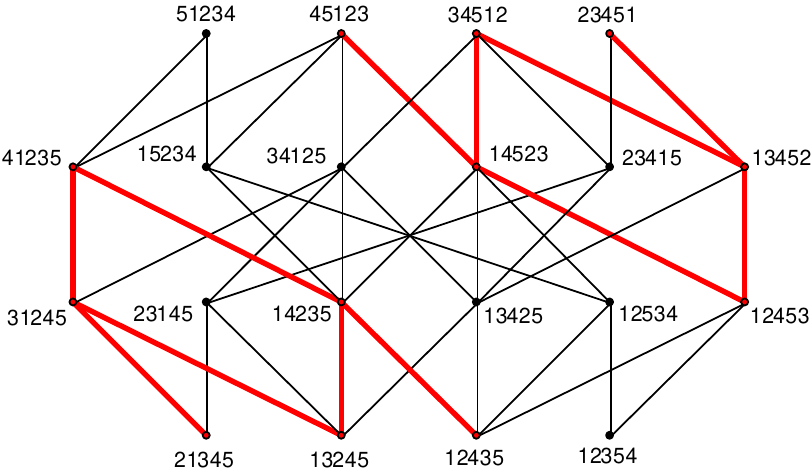}
\end{center}
\caption{The poset $(\bG_5,\le)$ with down-set of $41235$ and up-set of $12453$ highlighted.}
\label{B5Poset}
\end{figure}

We mention that $\(\bG_n,\le\)$ does not embed into $\(\frakS_n,\le\)$ while preserving the rank-function (recall that the rank-function for $\(\frakS_n,\le\)$ is given by the number of inversions). 
As an example, the incomparable BG permutations $\pi=41235\cdots n$ and $\sigma=34125\cdots n$ have respective length vectors $\bl\(\pi\)=(1,3,1,n-3)$ and $\bl\(\sigma\)=(1,2,2,n-3)$, and are at the same rank-level $\lambda\(\pi\)=\lambda\(\sigma\)=2$ in $\(\bG_n,\le\)$. However, as $\inv\(\pi\)=3$ and $\inv\(\sigma\)=4$, they reside at different rank-levels in $\(\frakS_n,\le\)$.

\section{Structural Properties}\label{sec-structure}

In this section, we present several structural features of $(\bG_n,\le)$ that are consequences of Theorem \ref{thm-BGcompCrit}.
For $\pi,\sigma\in\bG_n$, we say that $\sigma$ \emph{covers} $\pi$ and write ``$\pi\vtl\sigma$'' provided that $\pi<\sigma$ and there is no intermediate element $\rho\in\bG_n$ with $\pi<\rho<\sigma$. 

Several properties of the Hasse diagram for $\(\bG_n,\le\)$ appear in Reading \cite[\S 8]{Reading}.  In particular, the number of elements covering (or covered by) a fixed $\sigma \in \bG_n$, the total number of covering relations, the number of minimal and maximal elements, the rank function, and the number of elements at a fixed rank are explicit or implicit in the discussion presented there. For completeness, we state these in terms of our comparability criterion in Corollaries \ref{cor-degree} and \ref{cor-ranked}. (Reading also considered antichains in $\bG_n$, showing that the size of the largest antichain is $\lfloor \frac{n}{2} \rfloor \lceil \frac{n}{2} \rceil$ by providing a symmetric chain decomposition \cite{Reading}.)

\begin{corollary}\label{cor-degree}
 In $(\bG_n,\le)$, each element is covered by 0, 2, or 4 elements, and each element covers 0, 2, or 4 elements.  There are $4\binom{n}{3}$ covering relations in $(\bG_n,\le)$, $n>2$.
\end{corollary}

\begin{proof}
Elements covered by $\sigma$ have $1$ subtracted from the second or third entry in the length vector and added to the first or fourth entry, elements covering $\sigma$ have $1$ subtracted from the first or fourth entry and added to the second or third entry, and all entries must be positive.

Thus for fixed $\sigma$ with length vector $(\ell_1,\ell_2,\ell_3,\ell_4)$, the total number of elements $\sigma$ covers is either $4$, $2$, or $0$. More precisely $\sigma$ covers a total of $$4 - 2\cdot\textbf{1}_{\{\ell_2=1\}} - 2\cdot\textbf{1}_{\{\ell_3=1\}}$$ elements. Summing this quantity over all possible $\sigma$ we get that the total number of covering relations is 
\[
4\binom{n+1}{3} - 2\binom{n}{2} - 2\binom{n}{2} = 4\binom{n}{3}.\qedhere
\]
\end{proof}

It is also clear that there are $n-1$ minimal and $n-1$ maximal elements in $(\bG_n,\le)$, as these require either the two middle coordinates to be $1$ (minimal) or the two outer coordinates to be $1$ (maximal). 

\begin{corollary}\label{cor-ranked}
$(\bG_n,\le)$ is a ranked poset with rank-function given by 
$\lambda\(\sigma\):=\ell_2+\ell_3-2.$
Given $0\le k \le n-2$, there are $(k+1)(n-k-1)$ elements of $(\bG_n,\le)$ with rank-level $k$.
\end{corollary}

\begin{proof}
The first statement comes directly from the comparability criterion. To count the number of $\sigma\in\bG_n$ with $k=\lambda\(\sigma\)=\ell_2+\ell_3-2$, we must count the number of ways to satisfy $\ell_2+\ell_3=k+2$ and $\ell_1+\ell_4=n-k$, with each $\ell_i\ge 1$. There are $k+1$ solutions to the first equation, and $n-k-1$ solutions to the second. 
\end{proof}

\begin{remark}
Corollary \ref{cor-ranked} provides a bijective proof of the identity $\sum_{i=1}^{n-1}i(n-i)=\binom{n+1}{3}$. Indeed, the right side is the number of BG permutations of order $n$, and the left side counts these same permutations by their rank-level $k=i-1$ for $1\le i \le n-1$. For a geometric version of this interpretation of the binomial identity, see \cite[\S 3]{Propp}.
\end{remark}

The remainder of this section presents new structural results for $\bG_n$. We begin with a definition.

\begin{definition}
A \emph{saturated $r$-chain} in $(\bG_n,\le)$ is a set of elements $\pi_1,\ldots,\pi_r \in \bG_n$ with $\pi_1 \vtl \cdots \vtl \pi_r$.
\end{definition}

\noindent We count the number of saturated $r$-chains between two fixed elements as well as the number of \emph{maximal chains} in $(\bG_n,\le)$, i.e., the number of saturated chains connecting a minimal and maximal pair of elements in the BG poset.  These are saturated $(n-1)$-chains in $(\bG_n,\le)$ by Corollary \ref{cor-ranked}.

\begin{theorem}\label{thm-number of r chains}
Let $\pi,\sigma \in \bG_n$ have respective length vectors $(\ell_1,\ell_2,\ell_3,\ell_4)$ and \\
$(m_1,m_2,m_3,m_4)$, and suppose $\pi \leq \sigma$. Let $r = \ell_1-m_1+\ell_4-m_4$. Then there are $\binom{r}{\ell_1-m_1}\binom{r}{m_2-\ell_2}$ saturated $r$-chains between $\pi$ and $\sigma$.  There are $4^{n-2}$ maximal chains in $\(\bG_n,\le\)$.
\end{theorem}

\begin{proof}
The saturated chains between $\pi$ and $\sigma$ all have length $r$, as this is the total amount to remove from the outer coordinates $\{\ell_1,\ell_4\}$ and add to the inner coordinates $\{\ell_2,\ell_3\}$ in such a way that the resulting length vector is $(m_1,m_2,m_3,m_4)$. Moreover, there are $\binom{r}{\ell_1-m_1}$ ways to choose the order to remove things from $\ell_1$ and $\ell_4$, and $\binom{r}{m_2-\ell_2}$ ways to add things to $\ell_2$ and $\ell_3$. This gives $\binom{r}{\ell_1-m_1}\binom{r}{m_2-\ell_2}$ saturated $r$-chains. For any pair of minimal and maximal elements $\pi,\sigma\in\bG_n$, with respective length vectors $\bl(\pi)=(\ell_1,1,1,\ell_4)$ and $\bl(\sigma)=(1,m_2,m_3,1)$, notice that we have $\pi\le\sigma$ by Theorem \ref{thm-BGcompCrit}. 
Then the number of maximal chains is
\[
\sum_{\substack{1\le \ell_1\le n-1 \\ 1\le m_2\le n-1 }}\binom{n-2}{\ell_1-1}\binom{n-2}{m_2-1}=\(\sum_{k=0}^{n-2}\binom{n-2}{k}\)^2=4^{n-2}.\qedhere
\]
\end{proof}

We also find the distance between two fixed elements $\pi,\sigma\in\bG_n$ in the Hasse diagram. This is a minimal length \emph{Hasse walk} between $\pi$ and $\sigma$, i.e. a sequence of BG elements $\pi=\pi_1,\pi_2,\ldots,\pi_r=\sigma$ such that either $\pi_{i+1}$ covers $\pi_i$ or vice versa, where $1\le i <r$ and $r$ is minimal. 

\begin{theorem}\label{thm-conn}
Let $\pi, \sigma \in \bG_n$ have respective length vectors $(\ell_1,\ell_2,\ell_3,\ell_4)$ and \\
$(m_1,m_2,m_3,m_4)$. Then the distance between $\pi$ and $\sigma$ in the Hasse diagram for $\bG_n$ is 
\[
\max\{|\ell_1-m_1|+|\ell_4-m_4|,|\ell_2-m_2|+|\ell_3-m_3|\}. 
\]
In particular, the Hasse diagram for $\bG_n$ is connected.
\end{theorem}

\begin{proof}
Since moving along an edge in the Hasse diagram changes exactly one inner and one outer coordinate by $1$, the distance is at least this maximum value. To show that this value is attained, suppose first that $|\ell_1-m_1|+|\ell_4-m_4|$ is the maximum value.  Assume that $|\ell_1-m_1|\geq |\ell_4-m_4|$ and $\ell_1 > m_1$. Iteratively remove $\ell_1-m_1$ from the first coordinate of $\pi$ and add to either of $\ell_2$ or $\ell_3$ that is less than $m_2$ or $m_3$, respectively (if both are at least $m_2$ or $m_3$, respectively, then add to either coordinate). After $\ell_1-m_1$ steps we have $(m_1,n_2,n_3,\ell_4)$.  Repeating this procedure with the fourth coordinate (if $\ell_4<m_4$ we subtract from $n_2$ and $n_3$) gives the result. The case where $|\ell_2-m_2|+|\ell_3-m_3|$ is the maximum value is similar. 
\end{proof}

The poset of BG permutations is not a lattice.  In fact, we can enumerate the exact number of \emph{lattice-obstructions} in the Hasse diagram for $\bG_n$. By this we mean a set $\{\pi_1,\pi_2,\sigma_1,\sigma_2\}\subseteq\bG_n$ with $\pi_1$ and $\pi_2$ having the same rank-level $k$ (and so are incomparable), $\sigma_1$ and $\sigma_2$ having rank-level $k+1$, and $\pi_i \vtl \sigma_j$, $i,j\in[2]$. Such a substructure is, indeed, a lattice-obstruction since $\sigma_1$ and $\sigma_2$ will have no infimum. Likewise, $\pi_1$ and $\pi_2$ will have no supremum. For brevity, let us refer to these lattice-obstructions as \emph{butterflies}, since they resemble a butterfly; see $(\bG_3,\le)$ in Figure \ref{BGPoset} for an example.

\begin{theorem}\label{thm-butterfly}
There are $\binom{n}{3} + \binom{n-2}{3}$ butterflies in the Hasse diagram for $\bG_n$. 
Furthermore, each Hasse edge $\pi \vtl \sigma$ is in either one or two butterflies.  The edge $\pi \vtl \sigma$ is in a unique butterfly if and only if there is some fixed coordinate $i \in [4]$ such that $\ell_i(\pi) = \ell_i(\sigma) = 1$. 
\end{theorem}

\begin{proof}

Start with $\pi_1$ with length vector $(\ell_1,\ell_2,\ell_3,\ell_4)$. The $\sigma_j$ must be obtained by lowering one of $\ell_1$ or $\ell_4$ and raising one of $\ell_2$ or $\ell_3$.  Since $\pi_1 \vtl \sigma_1$, we know that at least one of $\ell_1$ and $\ell_4$ is larger than $1$.  If exactly one of $\ell_1$ or $\ell_4$ is equal to $1$, then without loss of generality the length vector is $(1,\ell_2,\ell_3,\ell_4)$ and note that there are exactly two elements of $\bG_n$ covering $\pi_1$.  There is one possible $\pi_2$, namely that with length vector $(2,\ell_2,\ell_3,\ell_4-1)$.  If $\ell_1$ and $\ell_4$ are both larger than $1$ and $\ell_2=\ell_3=1$, then there are two possible $\pi_2$.  If $\ell_1$ and $\ell_4$ are both larger than $1$ and exactly one of $\ell_2$ and $\ell_3$ is equal to $1$, then there are three possible $\pi_2$.  If all coordinates are larger than $1$, then there are four possible $\pi_2$.  See Table \ref{fig-butterfly} for the possible $\pi_2$ as well as the elements covering both $\pi_1$ and $\pi_2$.

\begin{table}[ht]
\begin{center}
\begin{tabular}{c|c|c}
$\pi_1$ & $\pi_2$ & $\sigma_1$, $\sigma_2$ \\
\hline
$(1,\ell_2,\ell_3,\ell_4)$ & $(2,\ell_2,\ell_3,\ell_4-1)$ & $(1,\ell_2+1,\ell_3,\ell_4-1)$, $(1,\ell_2,\ell_3+1,\ell_4-1)$\\
\hline
$(\ell_1,1,1,\ell_4)$ & $(\ell_1-1,1,1,\ell_4+1)$ & $(\ell_1-1,2,1,\ell_4)$, $(\ell_1-1,1,2,\ell_4)$\\
& $(\ell_1+1,1,1,\ell_4-1)$ & $(\ell_1,2,1,\ell_4-1)$, $(\ell_1,1,2,\ell_4-1)$\\
\hline
$(\ell_1,1,\ell_3,\ell_4)$ & $(\ell_1-1,1,\ell_3,\ell_4+1)$ & $(\ell_1-1,2,\ell_3,\ell_4)$, $(\ell_1-1,1,\ell_3+1,\ell_4)$ \\
& $(\ell_1+1,1,\ell_3,\ell_4-1)$  & $(\ell_1,2,\ell_3,\ell_4-1)$, $(\ell_1,1,\ell_3+1,\ell_4-1)$\\
& $(\ell_1,2,\ell_3-1,\ell_4)$ & $(\ell_1-1,2,\ell_3,\ell_4)$, $(\ell_1,2,\ell_3,\ell_4-1)$\\
\hline
$(\ell_1,\ell_2,\ell_3,\ell_4)$ & $(\ell_1-1,\ell_2,\ell_3,\ell_4+1)$ & $(\ell_1-1,\ell_2+1,\ell_3,\ell_4)$, $(\ell_1-1, \ell_2,\ell_3+1,\ell_4)$\\
& $(\ell_1+1,\ell_2,\ell_3,\ell_4-1)$ & ($\ell_1,\ell_2+1,\ell_3,\ell_4-1)$, $(\ell_1,\ell_2,\ell_3+1,\ell_4-1)$\\
& $(\ell_1,\ell_2+1,\ell_3-1,\ell_4)$ & $(\ell_1-1,\ell_2+1,\ell_3,\ell_4)$, ($\ell_1,\ell_2+1,\ell_3,\ell_4-1)$ \\
& $(\ell_1,\ell_2-1,\ell_3+1,\ell_4)$ & $(\ell_1-1,\ell_2,\ell_3+1,\ell_4)$, $(\ell_1,\ell_2,\ell_3+1,\ell_4-1)$\\
\end{tabular}
\end{center}
\caption{The possible $\pi_2$ for a fixed $\pi_1$, and the corresponding $\sigma_1$ and $\sigma_2$.  The unspecified values $\ell_i$ are all larger than $1$.} 
\label{fig-butterfly}
\end{table}

We now need to count the number of permutations satisfying each of these cases. There are $\binom{n-1}{2}$ vectors with only $\ell_1=1$ and the same number with only $\ell_4=1$.  There are $n-3$ with only $\ell_2$ and $\ell_3$ being $1$. There are $\binom{n-3}{2}$ with only $\ell_2=1$ and $\binom{n-3}{2}$ with only $\ell_3=1$.  There are then $\binom{n-3}{3}$ with all entries larger than $1$. Dividing by two for the overcount gives 
\[ \frac{1}{2} \l[2\binom{n-1}{2} + 2\cdot(n-3) + 3\cdot 2\binom{n-3}{2} + 4\cdot \binom{n-3}{3} \r] = \binom{n}{3} + \binom{n-2}{3}\]
total butterflies in $(\bG_n,\le)$.
\end{proof}

\noindent Theorem \ref{thm-butterfly} says that \emph{every} Hasse arc in $(\bG_n,\le)$ participates in a lattice-obstruction. 
Since $(\bG_n,\le)$ is not a lattice, not every collection of elements will have an infimum (supremum, respectively). This leads to the following definition.

\begin{definition}
Let $\pi_1,\ldots,\pi_r \in \bG_n$.  We call $\sigma \in \bG_n$ a \emph{maximal element below $\pi_1,\ldots,\pi_r$} if $\sigma \leq \pi_i$ for each $i \in [r]$, and if $\tau \in \bG_n$ satisfies $\tau \leq \pi_i$ for $i \in [r]$ and $\tau \geq \sigma$ then $\tau = \sigma$.
\end{definition}

\noindent Note that there may be zero, one, or more than one such $\sigma$; if there are multiple $\sigma$ then those elements must be incomparable.

\begin{theorem}\label{thm-inf}
Let $\pi_1,\ldots,\pi_r \in \bG_n$. An element $\sigma \in \bG_n$ is a maximal element below $\pi_1,\ldots,\pi_r$ if and only if $\ell_1(\sigma) \geq \max_{1\leq j \leq r} \ell_1(\pi_j)$, $\ell_2(\sigma) \leq \min_{1\leq j \leq r} \ell_2(\pi_j)$, $\ell_3(\sigma) \leq \min_{1\leq j \leq r} \ell_3(\pi_j)$,  $\ell_4(\sigma) \geq \max_{1 \leq j \leq r} \ell_4(\pi_j)$, with equality in $\ell_2$ and $\ell_3$ or with equality in $\ell_1$ and $\ell_4$.
\end{theorem}

\begin{proof}
Fix $\pi_1,\ldots,\pi_r \in \bG_n$.  If $\sigma$ is a maximal element below $\pi_1,\ldots,\pi_r$, we will show that the conditions on $\ell_i(\sigma)$ for $i \in [4]$ must hold. By Theorem \ref{thm-BGcompCrit}, it is clear that 
\begin{eqnarray*}
&\ell_1(\sigma) \geq \max_{1\leq j \leq r} \ell_1(\pi_j), \ell_2(\sigma) \le \min_{1\leq j \leq r} \ell_2(\pi_j), \\
&\ell_3(\sigma) \le \min_{1\leq j \leq r} \ell_3(\pi_j), \textrm{ and } \ell_4(\sigma) \geq \max_{1 \leq j \leq r} \ell_4(\pi_j).
\end{eqnarray*}
Furthermore, the equality conditions must be met, since if (without loss of generality) $\ell_1(\sigma) > \max_{1\leq j \leq r} \ell_1(\pi_j)$ and $\ell_2(\sigma) < \max_{1 \leq j \leq r} \ell_2(\pi_j)$, then $\tau \in \bG_n$ with length vector $(\ell_1(\sigma)-1,\ell_2(\sigma)+1,\ell_3(\sigma),\ell_4(\sigma))$ satisfies $\sigma < \tau$ and $\tau \leq \pi_j$ for $j \in [r]$.

For the converse, suppose that the conditions on $\ell_i(\sigma)$ for $i \in [4]$ hold as well as one of the equality conditions, Then $\sigma \leq \pi_j$ for each $j \in [r]$ by Theorem \ref{thm-BGcompCrit}. 
Suppose that $\tau \in \bG_n$ is such that $\tau \leq \pi_j$ for $j \in [r]$ and $\sigma \leq \tau$. Then by Theorem \ref{thm-BGcompCrit} we have $\ell_1(\tau) \geq \ell_1(\pi_j)$, $\ell_2(\tau) \leq \ell_2(\pi_j)$, $\ell_3(\tau) \leq \ell_3(\pi_j)$, and $\ell_4(\tau) \geq \ell_4(\pi_j)$ for all $j \in [r]$.  
But $\sigma \leq \tau$ implies that the equality conditions for $\sigma$ are also equality for $\tau$, and this forces $\sigma$ and $\tau$ to have the same rank. Since they have the same rank and are comparable, they must be equal.
\end{proof}

A way to view Theorem \ref{thm-inf} is the following.  Let $\pi_1,\ldots,\pi_r\in\bG_n$, and calculate the value
\[
\max_{1 \leq j \leq r} \ell_1(\pi_j) + \min_{1\leq j \leq r} \ell_2(\pi_j) + \min_{1\leq j \leq r} \ell_3(\pi_j) +  \max_{1 \leq j \leq r} \ell_4(\pi_j).
\]
If this expression is equal to $n+2$, then we define the unique $\sigma$ by defining $\ell_i(\sigma)$ to be the corresponding minimum or maximum value.  If this expression is smaller than $n+2$,  add to the outer coordinates until the sum equals $n+2$; all possible ways of adding to the outer coordinates gives the possible maximal elements below $\pi_1,\ldots,\pi_r$. In this case, there is equality in $\ell_2$ and $\ell_3$.  If this expression is greater than $n+2$, remove from the values in the inner coordinates until the sum reaches $n+2$; note that this might not be possible since the inner coordinates must be positive. 
On the other hand, there may be multiple ways to do this as well, all of which deliver a maximal element below $\pi_1,\ldots,\pi_r$. 
When this can be done, there is equality in $\ell_1$ and $\ell_4$.

All elements that are maximal elements below $\pi_1,\ldots,\pi_r$ have the same rank. Under the involution $f_{2143}$ on $\bG_n$, an analogous result holds for the similarly defined \emph{minimal element above $\pi_1,\ldots,\pi_r$}. We leave the details of this extension to the interested reader.

We next consider the M\"obius function on intervals of $\bG_n$. For $\pi,\sigma \in \bG_n$ we denote the interval $[\pi,\sigma] := \{\tau : \pi \leq \tau \leq \sigma\}$ and then we can recursively define the M\"obius function on intervals by $\mu(\pi,\pi)=1$ and $\mu(\pi,\sigma) = -\sum_{\pi \leq \tau < \sigma} \mu(\pi,\tau)$ (see, e.g., \cite{Stanley1}). Here an empty sum is $0$, so if $\pi \nleq \sigma$ then $\mu(\pi,\sigma)=0$. The M\"obius function has been studied on certain subposets of the permutation poset; see, e.g., \cite{Be1,Be2,Bj,W}.

For $\bG_n$, we are able to determine the M\"obius function on every interval; we will show that it takes values in the set $\{-2,-1,0,1,3,4\}$.  We make a few definitions and remarks before stating and proving the theorem. Suppose that $\pi, \sigma \in \bG_n$ have respective length vectors $(\ell_1,\ell_2,\ell_3,\ell_4)$ and $(m_1,m_2,m_3,m_4)$, and $\pi \leq \sigma$.  We let $d_i = |\ell_i-m_i|+1$ for $i \in [4]$. 

We first note that if $\pi_1 \leq \sigma_1$ and $\pi_2 \leq \sigma_2$ have the same values of $d_i$ for $i \in [4]$, then the intervals $[\pi_1,\sigma_1]$ and $[\pi_2,\sigma_2]$ are isomorphic, as the length vectors that lie in each interval are linear translates of each other. So it suffices to assume that $\pi$ has length vector $(x_1,1,1,x_4)$ and $\sigma$ has length vector $(1,x_2,x_3,1)$ (that is, they are minimal and maximal elements in $\bG_k$ for $k=x_1+x_4=x_2+x_3$). In this case we have $d_i = x_i$ for each $i \in [4]$.  The M\"obius function will be defined entirely in terms of the values $d_i$.

Second, we remark that for a $\tau$ with $\pi \leq \tau \leq \sigma$, there is a bijection between the \emph{distance vector from $\pi$ for $\tau$} $(d_1,d_2,d_3,d_4)$ and the length vector $(y_1,y_2,y_3,y_4)$ for $\tau$.  Also notice that for any distance vector from $\pi$ for $\tau$ we have $d_i \geq 1$ for $i \in [4]$,  $d_1+d_4=d_2+d_3$, and $d_1+d_4-2$ is the difference in the ranks of $\pi$ and $\tau$.

\begin{example}\label{ex-Mobius}
For $\pi_1$ with length vector $(6,3,4,9)$ and $\sigma_1$ with length vector $(2,8,7,5)$, the distance vector from $\pi_1$ for $\sigma_1$ is given by $(d_1,d_2,d_3,d_4)=(5,6,4,5)$. This is the same distance vector from $\pi_2$ with length vector $(5,1,1,5)$ for $\sigma_2$ with length vector $(1,6,4,1)$.  The interval $[\pi_1,\sigma_1]$ is isomorphic to the interval $[\pi_2,\sigma_2]$.
\end{example}

We are now ready to compute the M\"obius function for any interval in $\bG_n$.

\begin{theorem}\label{thm-Mobius}
Let $\pi, \sigma \in \bG_n$ have respective length vectors $(\ell_1,\ell_2,\ell_3,\ell_4)$ and $(m_1,m_2,m_3,m_4)$, and $\pi \leq \sigma$.  Let $d_i = |\ell_i-m_i|+1$ for $i \in [4]$. Then we have the following:
\begin{enumerate}
    \item\label{case1} If $d_1=d_2=d_3=d_4=1$, then $\mu(\pi,\sigma)=1$;
    \item\label{case2} If $d_1+d_4=d_2+d_3=3$, then $\mu(\pi,\sigma)=-1$;
    \item\label{case3} If $d_1=d_2=d_3=d_4=2$, then $\mu(\pi,\sigma)=3$;
    \item\label{case4} If $d_1=d_2=d_3=d_4=a\geq 3$, then $\mu(\pi,\sigma)=4$;
    \item\label{case5} If $a \geq 2$ with $d_1=d_4=a$ and $d_2=a \pm 1$ and $d_3=a\mp 1$, or if $a \geq 2$ with $d_2=d_3=a$ and $d_1=a\pm 1$ and $d_4=a\mp 1$, then $\mu(\pi,\sigma)=1$;
    \item\label{case6} If $a \geq 2$ with exactly one of $d_1$ and $d_4$ equaling $a$, exactly one of $d_2$ and $d_3$ equaling $a$, and the other two values equaling $a+1$, then $\mu(\pi,\sigma)=-2$; and
    \item If none of the above holds, then $\mu(\pi,\sigma)=0$.
\end{enumerate}
\end{theorem}

In short, if the distance vector entries are not ``almost equal,'' then the M\"obius function on that interval will be zero, and if they are ``almost equal'' then the M\"obius function on that interval will be non-zero. Note that the first several cases generally cover small values of $n$ (when viewing the interval as between a minimal and maximal element in $\bG_n$ for some $n$), and the last several cases cover larger values of $n$. 
To illustrate Theorem \ref{thm-Mobius}, when using the elements of $\bG_n$ given in Example \ref{ex-Mobius} we have $\mu(\pi_1,\sigma_1)=\mu(\pi_2,\sigma_2) = 1$ by part \ref{case5}. Furthermore, we remark that if $n \geq 6$ then all possible cases of the distance vectors are possible for some interval in $\bG_n$, and so all values of $\{-2,-1,0,1,3,4\}$ appear  as a value of $\mu(\pi,\sigma)$ for some choice of $\pi,\sigma \in \bG_n$, and intervals that are ``long'' have value $-2$, $1$, $4$, or $0$.

To prove the theorem, we will consider a fixed element $\sigma'$ which is covered by $\sigma$.  Many $\tau$ satisfy $\pi \leq \tau \leq \sigma'$, and by definition of the M\"obius function we have that $\sum_{\pi \leq \tau \leq \sigma'} \mu(\pi,\tau) = 0$.  We then separately consider the value of $\mu(\pi,\tau)$ where $\tau \nleq \sigma'$.  These latter $\tau$ have strong restrictions on the coordinates of their distance vectors, and so most of their values are inductively equal to 0. Looking again at $\pi_2$ and $\sigma_2$ in Example \ref{ex-Mobius}, the proof will use $\sigma_2'$ with length vector $(2,5,4,1)$; all elements in the interval $[\pi_2,\sigma_2]$ that are not below $\sigma_2'$ must have a first coordinate of $1$ or a second coordinate of $6$, and so will have the largest possible first or second coordinate in their distance vector. 

Example \ref{ex-Mobiusproof}, which is given immediately after the proof, illustrates the calculations in the proof and may be worth reading in parallel with the proof.

\begin{proof}[Proof of Theorem \ref{thm-Mobius}]
Recall that we may assume $\pi$ has length vector $(x_1,1,1,x_4)$ and $\sigma$ has length vector $(1,x_2,x_3,1)$, and so the distance vector from $\pi$ for $\sigma$ is $(d_1,d_2,d_3,d_4)=(x_1,x_2,x_3,x_4)$.

To prove the theorem, we induct on $d_1+d_4$, which is two more than the difference in the ranks of $\pi$ and $\sigma$. When $d_1+d_4=2$, we have $\sigma=\pi$, and the result is clear. When $d_1+d_4 = 3$ and $d_1+d_4=4$, the result follows by inspection on the minimal and maximal elements of $\bG_3$ and $\bG_4$ (see Figure \ref{BGPoset}). So we assume that $d_1+d_4>4$. We assume without loss of generality that $d_1 \geq d_4$ and $d_2 \geq d_3$ (the arguments are similar in the other cases, as the proof will simply use the largest value in the inner and outer coordinates); note that $d_1$ and $d_2$ must be at least $3$. 

Consider $\sigma'$ with length vector $(2,x_2-1,x_3,1)$. 
By induction (so that the M\"obius function is defined on $[\pi,\sigma']$) and the definition of the M\"obius function, we have 
\begin{equation}\label{eq-sigma'}
\sum_{\pi \leq \tau \leq \sigma'} \mu(\pi,\tau) = 0.
\end{equation}
To compute $\mu(\pi,\sigma)$, we also need to compute $\mu(\pi,\tau)$ where $\tau$ has $\pi \leq \tau < \sigma$ and $\tau \nleq \sigma'$, which means that $\tau$ has length vector $(y_1,y_2,y_3,y_4)$ satisfying $y_1=1$ or $y_2=x_2$.  In terms of the distance vector from $\pi$ for $\tau$, we have $d_1=x_1$ or $d_2=x_2$. In light of (\ref{eq-sigma'}), we call a $\tau$ that meets one of these two conditions one that \emph{contributes to the M\"obius function}. 
We next analyze these $\tau$ more closely based on the conditions required.

Suppose that $y_1=1$, i.e., that the first coordinate of the distance vector from $\sigma$ for $\tau$ satisfies $d_1=x_1$. Recall that by assumption we have $x_1 \geq x_4$. Since $\tau < \sigma$ we have the sum of the outer coordinates of the distance vector from $\pi$ for $\tau$ must be smaller than $x_1+x_4$, i.e., $d_1+d_4<x_1+x_4$.  Since $d_1=x_1$, we have $d_4<x_4\leq x_1$. In short, the value $d_4$ in the distance vector from $\pi$ for $\tau$ is smaller than $x_1$.  By induction the only non-zero values of $\mu(\pi,\tau)$ occur when $\tau$ has ``almost equal'' values for each $d_i$. As $d_1+d_4=d_2+d_3$, in this case we have non-zero values of $\mu(\pi,\tau)$ for $\tau$ with:
\begin{itemize}
\item $d_4=x_1 - 2$ (and $d_2=d_3=x_1-1$; here we use that $d_4<x_1$); or
\item $d_4=x_1 -1$ with exactly one of $d_2$ or $d_3$ having value $x_1\pm 1$ and the other having value $x_1$ (here again we use that $d_4<x_1$).
\end{itemize}

Suppose next that $y_2=x_2$, i.e., that that the second coordinate of the distance vector from $\pi$ for $\tau$ satisfies $d_2=x_2$. As before, $d_2+d_3<x_2+x_3$ combined with $d_2=x_2$ and $x_2 \geq x_3$ imply that $d_3<x_2$, where these $d_i$ values are part of the distance vector from $\pi$ for $\tau$.  Again, by induction the only non-zero values of $\mu(\pi,\tau)$ occur when $\tau$ has ``almost equal'' values for each $d_i$, and so we have non-zero values of $\mu(\pi,\tau)$ for $\tau$ with:
\begin{itemize}
\item $d_3=x_2 - 2$ (and $d_1=d_4=x_2-1$; here we use that $d_3<x_2$); or
\item $d_3=x_2-1$ with exactly one of $d_1$ or $d_4$ having value $x_2- 1$ and the other having value $x_2$ (here again we use that $d_3<x_2$). 
\end{itemize}

We now have the tools to compute the M\"obius function values based on the values in $(x_1,x_2,x_3,x_4)$ with $x_1\geq x_4$ and $x_2 \geq x_3$ arising from the initial choice of $\pi$ and $\sigma$. Recall that we are assuming that $x_1+x_4>4$ (and so the sum of the outer coordinates of the distance vector from $\pi$ to $\sigma$ is greater than $4$, i.e. $d_1+d_4>4$) and $x_1,x_2 \geq 3$. We have several cases to consider.  

\textbf{Case:} Suppose that $x_1=x_2=x_3=x_4$. 
We need to analyze those $\tau$ which contribute to the M\"obius function, so first assume that we consider $\tau$ and its distance vector $(d_1,d_2,d_3,d_4)$ from $\pi$ under the restriction of $d_1=x_1$.   By the previous analysis, there is such a $\tau$ with $d_4=x_1-2$ and $d_2=d_3=x_1-1$, giving $\mu(\pi,\tau)=1$.  And there is a $\tau$ with $d_4=x_1-1$ and $d_2=x_1-1$ and $d_3=x_1$, and another with $d_4=x_1-1$ and $d_2=x_1$ and $d_3=x_1-1$.  Each of these two $\tau$ gives $\mu(\pi,\tau)=-2$.

Similarly, in the case of $\tau$ with the restriction of $x_2=d_2$, there is such a $\tau$ with $d_3=x_2-2$ and $d_1=d_4=x_2-1$, giving $\mu(\pi,\tau) = 1$. We also have a $\tau$ with $d_3=x_2-1$ and $d_1=x_2-1$ and $d_4=x_2$, and another with $d_3=x_2-1$ and $d_1=x_2$ and $d_4=x_2-1$; this last one was already counted in the previous case as $x_1=x_2$. The former $\tau$ satisfies $\mu(\pi,\tau)=-2$.

In summary, the distance vectors from $\pi$ for $\tau$ that contribute a non-zero value to the M\"obius function are given by  $(x_1,x_1-1,x_1-1,x_1-2)$, $(x_1,x_1-1,x_1,x_1-1)$, and $(x_1,x_1,x_1-1,x_1-1)$;  from the second paragraph we also have $(x_2-1,x_2,x_2-2,x_2-1)$ and $(x_2,x_2,x_2-1,x_2-1)$.  The values of $\mu(\pi,\tau)$ for these five are $1$, $-2$, $-2$, $1$, and $-2$ (respectively). Note that by induction all other $\tau$ that contribue to the M\"obius function satisfy $\mu(\pi,\tau)=0$. Intuitively, in $\tau$ one coordinate of the distance vector is fixed, and so there are not many distance vectors whose coordinates are ``almost equal'' and so non-zero.

Then
\begin{align*}
\mu(\pi,\sigma) &= -\sum_{\pi \leq \tau < \sigma} \mu(\pi,\tau) \\ 
&= -\sum_{\pi \leq \tau \leq \sigma'} \mu(\pi,\tau) + \left(- \sum_{\pi \leq \tau < \sigma, \tau \nleq \sigma'} \mu(\pi,\tau)\right) = 0 - (1-2-2+1-2) = 4.    
\end{align*}

\textbf{Case:} Suppose that $x_1=x_4$ and $x_2=x_1\pm 1$ and $x_3=x_1\mp 1$. By our assumption that $x_2 \geq x_3$, we have that $x_2=x_1+1$ and $x_3=x_1-1$. Then the distance vectors from $\pi$ for $\tau$ that contribute a non-zero value to the M\"obius function are given by $(x_1,x_1,x_1-1,x_1-1)$ and $(x_1,x_1-1,x_1-1,x_1-2)$ with values of $\mu(\pi,\tau)$ given by $-2$ and $1$ (respectively). Again, all other $\tau$ that contribute to the M\"obius function satisfy $\mu(\pi,\tau)=0$ by induction.

Then
\begin{align*}
\mu(\pi,\sigma) &= -\sum_{\pi \leq \tau < \sigma} \mu(\pi,\tau) \\
&= -\sum_{\pi \leq \tau \leq \sigma'} \mu(\pi,\tau) + \left(- \sum_{\pi \leq \tau < \sigma, \tau \nleq \sigma'} \mu(\pi,\tau)\right) = 0 - (-2+1) = 1.    
\end{align*}

\textbf{Case:} Suppose that $x_2=x_3$ and $x_1=x_2\pm 1$ and $x_4=x_2\mp 1$. By our assumption that $x_1 \geq x_4$, we have that $x_1 = x_2+1$ and $x_4=x_2-1$. Then the distance vectors from $\pi$ for $\tau$ that contribute a non-zero value to the M\"obius function are given by $(x_2,x_2,x_2-1,x_2-1)$ and $(x_2-1,x_2,x_2-2,x_2-1)$ with values of $\mu(\pi,\tau)$ given by $-2$ and $1$ (respectively). Again, all other $\tau$ that contribute to the M\"obius function satisfy $\mu(\pi,\tau)=0$ by induction.

Then
\begin{align*}
\mu(\pi,\sigma) &= -\sum_{\pi \leq \tau < \sigma} \mu(\pi,\tau) \\
&= -\sum_{\pi \leq \tau \leq \sigma'} \mu(\pi,\tau) + \left(- \sum_{\pi \leq \tau < \sigma, \tau \nleq \sigma'} \mu(\pi,\tau)\right) = 0 - (-2+1) = 1.    
\end{align*}

\textbf{Case:} Suppose $x_1$ and $x_4$ differ by one and $x_2$ and $x_3$ differ by one. By our assumptions that $x_1 \geq x_4$ and $x_2 \geq x_3$, we have $x_1=x_2$, $x_4=x_1-1$, and $x_3=x_1-1$. Then the distance vectors from $\pi$ for $\tau$ that contribute a non-zero value to the M\"obius function are given by $(x_1,x_1-1,x_1-1,x_1-2)$ and $(x_1-1,x_1,x_1-2,x_1-1)$ with values of $\mu(\pi,\tau)$ given by $1$ and $1$. Again, all other $\tau$ that contribute to the M\"obius function satisfy $\mu(\pi,\tau)=0$ by induction.

Then
\begin{align*}
\mu(\pi,\sigma) &= -\sum_{\pi \leq \tau < \sigma} \mu(\pi,\tau) \\
&= -\sum_{\pi \leq \tau \leq \sigma'} \mu(\pi,\tau) + \left(- \sum_{\pi \leq \tau < \sigma, \tau \nleq \sigma'} \mu(\pi,\tau)\right) = 0 - (1+1) = -2.    
\end{align*}

\textbf{Case:} Suppose none of the previous cases apply. We will show that for any $\tau$ that contributes to the M\"obius function we have $\mu(\pi,\tau)=0$. Consider the distance vector $(d_1,d_2,d_3,d_4)$ from $\pi$ for $\tau$.  We imagine moving from the distance vector $(x_1,x_2,x_3,x_4)$ from $\pi$ to $\sigma$ to obtain the distance vector $(d_1,d_2,d_3,d_4)$ from $\pi$ to $\tau$. 

We know that such a $\tau$ either has first or second coordinate as its maximum value.  Suppose that the first coordinate has its maximum value.  Then since $\tau < \sigma$, the last coordinate must drop by at least one (in moving from the last coordinate in the distance vector from $\pi$ for $\sigma$). If the outer coordinates in the distance vector from $\pi$ for $\tau$ differ by exactly one, then those two coordinates were initially equal for $\sigma$.  Furthermore, in order for $\mu(\pi,\tau)\neq 0$ the inner coordinates for $\tau$ must differ by one. This implies that either they were initially equal or differed by two in $\sigma$, which puts us in one of the previous cases.

If the outer coordinates in $\tau$ differ by two, then they initially differed by at most one in $\sigma$. If they initially differed by zero in $\sigma$, then in order for $\mu(\pi,\tau)\neq 0$ we must have the inner coordinates of $\tau$ equal.  Therefore the inner coordinates of $\sigma$ must be equal or differ by two, which puts us in one of the previous cases.  If the outer coordinates differed by exactly one in $\sigma$, then in order for $\mu(\sigma,\tau)\neq 0$ the inner two coordinates of $\tau$ must be the same, which implies that the inner two coordinates of $\sigma$ differed by one.  This again puts us in one of the previous cases.

Similar arguments (essentially interchanging the role of inner and outer coordinates) show that if the second coordinate is maximum, then the distance vectors for any $\tau$ that contributes to the M\"obius function also has $\mu(\pi,\tau)=0$.  This shows that if the distance vector from $\pi$ to $\sigma$ does not satisfy any of the previous cases, then $\mu(\pi,\sigma)=0$.

Lastly, if $x_1<x_4$ or $x_2<x_3$, we simply interchange the role of $x_1$ and $x_4$ or $x_2$ and $x_3$ in the proof.
\end{proof}

We end with an example that illustrates many of the cases in the proof using specific intervals.

\begin{example}\label{ex-Mobiusproof}
Between $\pi=(5,1,1,5)$ and $\sigma=(1,5,5,1)$ we have $(x_1,x_2,x_3,x_4)=(5,5,5,5)$ and we group those below $\sigma'=(2,4,5,1)$ together.  The distance vectors for $\tau$ which correspond to non-zero values that contribute to the M\"obius function are  $(5,5,4,4)$ [value $\mu(\pi,\tau)=-2$], $(4,5,4,5)$ [value $\mu(\pi,\tau)=-2$], $(4,5,3,4)$ [value $\mu(\pi,\tau)=1$], $(4,4,5,5)$ [value $\mu(\pi,\tau)=-2$], and $(3,4,4,5)$ [value $\mu(\pi,\tau)=1$].  

Between $\pi=(5,1,1,5)$ and $\sigma=(1,4,6,1)$  we have $(x_1,x_2,x_3,x_4)=(5,4,6,5)$ and we group those below $\sigma'=(2,4,5,1)$ together.  The distance vectors for $\tau$ which correspond to non-zero values that contribute to the M\"obius function are $(5,4,5,4)$ [value $\mu(\pi,\tau)=-2$] and = $(5,4,4,3)$ [value $\mu(\pi,\tau)=1$].

Between $\pi=(4,1,1,5)$ and $\sigma=(1,5,4,1)$ we have $(x_1,x_2,x_3,x_4)=(4,5,4,5)$ and we group those below $\sigma'=(1,4,4,2)$ together. The distance vectors for $\tau$ which correspond to non-zero values that contribute to the M\"obius function are $(4,5,3,4)$ [value $\mu(\pi,\tau)=1$] and $(3,4,4,5)$ [value $\mu(\pi,\tau)=1$].

Between $\pi=(4,1,1,5)$ and $\sigma=(1,6,3,1)$ we have $(x_1,x_2,x_3,x_4)=(4,6,3,5)$ and we group those below $\sigma'=(1,5,3,2)$ together. All distance vectors for $\tau$ with second coordinate $6$ or last coordinate $5$ never have a non-zero value for $\mu(\pi,\tau)$.
\end{example}

\section{Probabilistic Properties}\label{sec-probability}
In the sections that follow, many of our results involve asymptotic analysis. For the various notations related to asymptotics, including $o(f(n)),O(f(n)),\Omega(f(n)),\Theta(f(n)),$ and $f(n)\sim g(n)$, we refer the reader to \cite{FlajoletSedgewick}. As a consequence of these analytical results, many of our proofs will take on a technical flavor, but the reason for pushing through these computations is that we are able to give rather precise results in the $\bG_n$ poset, which is not very often the case in these probabilistic analyses (see, e.g., \cite{Hammett,HammettPittel,Pittel3,Pittel2,Pittel1}). It is also our hope that these results could serve as the foundation upon which a more thorough analysis of this poset can be built. As we mentioned in the introduction, it may happen that a deeper understanding of the $\bG_n$ poset at this combinatorial probabilistic level may elicit insight into the lattice of monotone triangles (see \cite{LascouxSchutz} and \cite[ex. 7.103]{Stanley2}). 

The results in this section discuss the distribution of $\alpha$ and $\beta$ when restricted to $\bG_n$. 
Hereafter we regard $\alpha = \alpha(\w)$ and $\beta=\beta(\w)$ as random variables, where $\w\in\bG_n$ is selected uniformly at random from among all $\bg_n=\binom{n+1}{3}$ BG permutations. By Corollary \ref{cor-BGcompCrit}, $\alpha$ and $\beta$ are equidistributed. 

\begin{theorem}\label{thm-bgE[beta]}
For $\w\in \bG_n$ selected uniformly at random, we have
\[
\E\l[ \alpha \r] = \E\l[ \beta \r] = \frac{(n+3)(n+2)(n+1)}{60}= \frac{\bg_{n+2}}{10}
\]
and
\[
\Var\l[ \alpha \r] = \Var\l[ \beta \r] = \frac{(n+3)(n+2)^2(n+1)^2(n-2)}{2400}= \frac{3(n-2)\E\l[ \alpha \r]^2}{2(n+3)}.
\]
\end{theorem}
We remark that Balcza \cite{Balcza} computes the expectation and variance of $\beta$ for a uniformly random element of $\frakS_n$, the set of \emph{all} permutations.  Theorem \ref{thm-bgE[beta]} is essentially a special case of a more general formula for the product moments $\E\l[\alpha^k\beta^\ell\r]$, and so we delay the proof momentarily.  In this direction, define $g_k(t) := \sum_{i \geq 0} i^k t^i$ for each $k \geq 0$ (where $0^0:=1$).  It is known \cite[\S 1.4]{Stanley1} that 
\begin{equation}\label{eq-g_k(t)}
g_k(t) = \frac{A_k(t)}{(1-t)^{k+1}}, \quad k \geq 0,
\end{equation}
where $A_k(t) := \sum_{j=1}^{k} A_{k,j}t^j$ is the $k$th Eulerian polynomial (and we define $A_0(t)=1$). 
Furthermore, the Eulerian polynomials have bivariate exponential generating function  
\begin{equation}\label{eq-EulerianGF}
A(t,u):=\sum_{i \geq 0} A_i(t) \frac{u^i}{i!} = \frac{1-t}{1-te^{(1-t)u}} = (1-t) \sum_{i\geq 0} t^i e^{i(1-t)u}.
\end{equation}

We note that by (\ref{eq-EulerianGF}), for each $k\ge 0$ we have 
\begin{align}
D_k(t,u)&:= \frac{\partial^k}{\partial u^k}A(t,u)=\sum_{i\ge 0}A_{k+i}(t)\frac{u^i}{i!}\label{eq-EulerianGFderivative1}\\
&=(1-t)^{k+1}\sum_{i\ge 0}i^k t^i e^{i(1-t)u}   \nonumber\\
&=(1-t)^{k+1}g_k\(te^{(1-t)u}\)\label{eq-EulerianGFderivative2}.
\end{align}
We are now ready to prove the following.

\begin{theorem-E[a^kb^l]}
For each $k\ge 0$, $\ell\ge 0$, and $n\ge 2$ we have
\[
\E\l[\alpha^k\beta^\ell\r]=\frac{k!\ell!}{2^{k+\ell}\bg_n}\l[t^{n+2}u^kv^\ell\r]\frac{\(g_k\(te^{(1-t)u}\)-{\bf 1}_{\{k=0\}}\)^2\(g_\ell\(te^{(1-t)v}\)-{\bf 1}_{\{\ell=0\}}\)^2}{(1-t)^{k+\ell}}.
\]
\end{theorem-E[a^kb^l]}

\begin{proof} By Corollary \ref{cor-BGcompCrit} and (\ref{eq-g_k(t)}), we have 
{\footnotesize
\begin{align}
 &\E\l[ \alpha^k\beta^\ell\r] 
\nonumber    = \frac{1}{2^{k+\ell}\bg_n} \sum_{\substack{ w+x+y+z=n+2\\ w,x,y,z \geq 1}} w^k x^k(w+x)^k y^\ell z^\ell(y+z)^\ell\\
\nonumber    &= \frac{1}{2^{k+\ell}\bg_n} \sum_{\substack{0\le i \le k \\ 0 \le j \le \ell}}\binom{k}{i}\binom{\ell}{j}
                    \sum_{2\le m \le n}\(\sum_{\substack{ w+x=m\\ w,x \geq 1}} w^{k+i} x^{k+(k-i)}
                                    \sum_{\substack{ y+z=n+2-m\\ y,z \geq 1}} y^{\ell+j} z^{\ell+(\ell-j)}\)\\
\nonumber    &= \frac{1}{2^{k+\ell}\bg_n} \sum_{\substack{0\le i \le k \\ 0 \le j \le \ell}}\binom{k}{i}\binom{\ell}{j}
                \sum_{2\le m \le n}\l[t^m\r]\(g_{k+i}(t)-{\bf 1}_{\{k=0\}}\)\(g_{k+(k-i)}(t)-{\bf 1}_{\{k=0\}}\)\\
\nonumber& \hspace{2in} \cdot \l[t^{n+2-m}\r]\(g_{\ell+j}(t)-{\bf 1}_{\{\ell=0\}}\)\(g_{\ell+(\ell-j)}(t)-{\bf 1}_{\{\ell=0\}}\)\\
\nonumber    &= \frac{k!\ell!}{2^{k+\ell}\bg_n} \l[t^{n+2}\r]\(\sum_{0\le i \le k}\frac{\(g_{k+i}(t)-{\bf 1}_{\{k=0\}}\)}{i!}\frac{\(g_{k+(k-i)}(t)-{\bf 1}_{\{k=0\}}\)}{(k-i)!} \)\\
\nonumber&\hspace{2in} \cdot                                                        \(\sum_{0\le j \le \ell}\frac{\(g_{\ell+j}(t)-{\bf 1}_{\{\ell=0\}}\)}{j!}\frac{\(g_{\ell+(\ell-j)}(t)-{\bf 1}_{\{\ell=0\}}\)}{(\ell-j)!} \)\\
\label{eq-E[a^kb^l]}     &= \frac{k!\ell!}{2^{k+\ell}\bg_n} \l[t^{n+2}\r]\frac{1}{(1-t)^{3(k+\ell)+4}}\sum_{0\le i \le k}\frac{\(A_{k+i}(t)-(1-t){\bf 1}_{\{k=0\}}\)}{i!}\frac{\(A_{k+(k-i)}(t)-(1-t){\bf 1}_{\{k=0\}}\)}{(k-i)!}\\
\nonumber&\hspace{2in} \cdot
    \sum_{0\le j \le \ell}\frac{\(A_{\ell+j}(t)-(1-t){\bf 1}_{\{\ell=0\}}\)}{j!}\frac{\(A_{\ell+(\ell-j)}(t)-(1-t){\bf 1}_{\{\ell=0\}}\)}{(\ell-j)!} \\
\nonumber    &= \frac{k!\ell!}{2^{k+\ell}\bg_n} [t^{n+2}]\Bigg(\frac{1}{(1-t)^{3(k+\ell)+4}}\l[u^k\r]\(D_k(t,u)-(1-t){\bf 1}_{\{k=0\}}\)^2\l[v^\ell\r]\(D_\ell(t,v)-(1-t){\bf 1}_{\{\ell=0\}}\)^2\Bigg);
\end{align}}
here, in the last line we have used identity (\ref{eq-EulerianGFderivative1}). The result now follows from (\ref{eq-EulerianGFderivative2}).
\end{proof}

Using Theorem \ref{thm-E[a^kb^l]} we compute the asymptotic value of the product moments $\E\l[\alpha^k\beta^\ell\r]$.

\begin{theorem}\label{thm-E[a^kb^l]sim}
For each $k\ge 0$, $\ell\ge 0$, and $n\to \infty$ we have
\[
\E\l[\alpha^k\beta^\ell\r]\sim\frac{n^{3(k+\ell)}}{2^{k+\ell-1}(k+\ell+1)\binom{3(k+\ell)+2}{3k+1}(2k+1)\binom{2k}{k}(2\ell+1)\binom{2\ell}{\ell}}.
\]
\end{theorem}

\begin{proof}
Note that the result is clear for $(k,\ell) = (0,0)$. First assume $k,\ell \neq 0$. We write
\begin{equation}\label{eq-EulerGFproduct}
\sum_{0\le a\le b}\frac{A_{b+a}(t)}{a!}\frac{A_{b+(b-a)}(t)}{(b-a)!}:=\sum_{0\le i\le 3b}c_{b,i}t^i;
\end{equation}
here, the right-hand sum is obtained by expanding the left-hand sum of products, and collecting terms. Substituting (\ref{eq-EulerGFproduct}) into (\ref{eq-E[a^kb^l]})  delivers  
\begin{align}
\nonumber \E\l[\alpha^k\beta^\ell\r]
&= \frac{k!\ell!}{2^{k+\ell}\bg_n} \l[t^{n+2}\r]\frac{1}{(1-t)^{3(k+\ell)+4}}\sum_{0\le a \le 3k}c_{k,a}t^a\sum_{0\le b \le 3\ell}c_{\ell,b}t^b\\
\label{eq-E[a^kb^l]no1} &= \frac{k!\ell!}{2^{k+\ell}\bg_n} \sum_{\substack{0\le a \le 3k\\ 0\le b \le 3\ell}}c_{k,a}c_{\ell,b}\binom{n+2-(a+b)+3(k+\ell)+3}{3(k+\ell)+3}.
\end{align}
Thus, from (\ref{eq-EulerGFproduct}), (\ref{eq-E[a^kb^l]no1}), and $\bg_n = \binom{n+1}{3}$ we obtain 
\begin{align}
&\E\l[ \alpha^k\beta^\ell\r] \sim \frac{3!k!\ell!n^{3(k+\ell)}}{(3(k+\ell)+3)!2^{k+\ell}}\sum_{\substack{0\le a \le 3k\\ 0\le b \le 3\ell}}c_{k,a}c_{\ell,b}\nonumber\\
&=\frac{3!k!\ell!n^{3(k+\ell)}}{(3(k+\ell)+3)!2^{k+\ell}}\sum_{0\le i \le k}\frac{A_{k+i}(1)}{i!}\frac{A_{k+(k-i)}(1)}{(k-i)!}
                                                                    \sum_{0\le j \le \ell}\frac{A_{\ell+j}(1)}{j!}\frac{A_{\ell+(\ell-j)}(1)}{(\ell-j)!}\nonumber\\
&=\frac{3!k!^3\ell!^3n^{3(k+\ell)}}{(3(k+\ell)+3)!2^{k+\ell}}\sum_{0\le i \le k}\binom{i+k}{i}\binom{(k-i)+k}{(k-i)}
                                                                    \sum_{0\le j \le \ell}\binom{j+\ell}{j}\binom{(\ell-j)+\ell}{(\ell-j)}\label{Euler}\\
&=\frac{3!k!^3\ell!^3n^{3(k+\ell)}}{(3(k+\ell)+3)!2^{k+\ell}}\l[t^k\r]\(\frac{1}{(1-t)^{k+1}}\)^2\l[t^\ell\r]\(\frac{1}{(1-t)^{\ell+1}}\)^2\nonumber\\
&=\frac{3!k!^3\ell!^3n^{3(k+\ell)}}{(3(k+\ell)+3)!2^{k+\ell}}\binom{k+(2k+2)-1}{k}\binom{\ell+(2\ell+2)-1}{\ell}.\nonumber
\end{align}
Here, in (\ref{Euler}) we have used 
$A_r(t)=\sum_{\sigma\in\frakS_r}t^{1+\textrm{des}(\sigma)}, $
where $$\textrm{des}(\sigma)=|\{i\in [r-1]\, : \, \sigma(i)>\sigma(i+1)\}|$$ is the number of \emph{descents} of $\sigma$ \cite[\S 1.4]{Stanley1}. Hence we have $A_r(1)=|\frakS_r|=r!$, and after a little algebra the result follows.

If $k=0$ (and $\ell> 0$), then from (\ref{eq-E[a^kb^l]}) we have
\begin{align}
\nonumber \E\l[\beta^\ell\r]   &= \frac{\ell!}{2^{\ell}\bg_n} [t^{n+2}] \frac{1}{(1-t)^{3\ell+4}} \cdot t^2 \cdot \sum_{0\leq j \leq \ell} \frac{A_{\ell+j}(t)}{j!}\frac{A_{\ell+(\ell-j)}(t)}{(\ell-j)!}\\
\nonumber &= \frac{\ell!}{2^{\ell}\bg_n} [t^{n+2}] \frac{1}{(1-t)^{3\ell+4}} \cdot t^2 \cdot \sum_{0 \leq b \leq 3\ell} c_{\ell,b} t^b\\
\nonumber &= \frac{\ell!}{2^{\ell}\bg_n}  \sum_{0 \leq b \leq 3\ell} c_{\ell,b} \binom{n+2 - (2+b) + 3\ell + 3}{3\ell+3},
\end{align}
and so similar to the first computation we have
\begin{align}
\E\l[ \beta^\ell\r] &\sim \frac{3!\ell!n^{3\ell}}{(3\ell+3)!2^{\ell}}\sum_{ 0\le b \le 3\ell}c_{\ell,b}\nonumber\\
 &=  \frac{3!\ell!^3n^{3\ell}}{(3\ell+3)!2^{\ell}}\binom{\ell+(2\ell+2)-1}{\ell}\nonumber.
\end{align}
Again, after a little algebra the result follows.  If $\ell=0$ (and $k> 0$), then a similar computation (involving $k$) shows that the result holds in this case.
\end{proof}

\noindent Recalling that $\alpha$ and $\beta$ are equidistributed, this shows that for example $\E\l[\alpha^3\r] \neq \E\l[\alpha^2\beta\r]$.  We also immediately obtain: 

\begin{corollary}\label{thm-E[a^k]sim}
For each $k\ge 0$ and $n\to\infty$ we have 
\[
\E\l[ \alpha^k\r]=\E\l[ \beta^k\r] \sim \frac{n^{3k}}{2^{k-1}(k+1)(3k+2)(2k+1)\binom{2k}{k}}.\epf
\]
\end{corollary}

\noindent We next perform a few calculations that will prove useful.

\begin{example}\label{ex-E[a^kb^l]smallish}
Note that $g_0(t)=(1-t)^{-1}$, and so 
\begin{equation}\label{eq-E[a^0b^0]}
\l[u^0\r]\(g_0\(te^{(1-t)u}\)-1\)^2=\left. \(\frac{te^{(1-t)u}}{1-te^{(1-t)u}}\)^2\right|_{u=0}=\frac{t^2}{(1-t)^2}.
\end{equation}
Next, $g_1(t)=t(1-t)^{-2}$ and thus 
\begin{align}\label{eq-E[a^1b^0]}
\l[u^1\r]&\(g_1\(te^{(1-t)u}\)\)^2 =\l[u^1\r]\frac{t^2e^{2(1-t)u}}{\(1-te^{(1-t)u}\)^4}
    =\frac{\partial}{\partial u}\left.\( \frac{t^2e^{2(1-t)u}}{\(1-te^{(1-t)u}\)^4} \)\right|_{u=0} 
    =\frac{2t^2(1+t)}{(1-t)^4}.    
\end{align}
Finally, $g_2(t)=t(1+t)(1-t)^{-3}$ and a similar calculation gives 
\begin{align}\label{eq-E[a^2b^0]}
\l[u^2\r]\(g_2\(te^{(1-t)u}\)\)^2 
    &=\frac{2 t^2 \(t^4+10 t^3 +20 t^2 +10 t +1\)}{(1-t)^6}.    
\end{align}
\end{example}

\noindent With these results in hand, we now prove Theorem \ref{thm-bgE[beta]}.

\begin{proof}[Proof of Theorem \ref{thm-bgE[beta]}.] 
By (\ref{eq-E[a^0b^0]}) and (\ref{eq-E[a^1b^0]}) 
\begin{align*}
\E\l[ \alpha^1\beta^0\r]&=\frac{1}{2\bg_n}\l[ t^{n+2}\r]\frac{1}{(1-t)}\(\l[u^1\r]\(g_1\(te^{(1-t)u}\)\)^2 \l[v^0 \r]\(g_0\(te^{(1-t)v}\)-1\)^2\)\\
    &=\frac{1}{2\bg_n}\l[ t^{n+2}\r]\frac{2t^4(1+t)}{(1-t)^7} = \frac{1}{\bg_n}\(\binom{n-2+7-1}{n-2}+\binom{n-3+7-1}{n-3}\)\\
    &=\frac{(n+3)(n+2)(n+1)}{60}.
\end{align*}
Similarly, by (\ref{eq-E[a^0b^0]}) and (\ref{eq-E[a^2b^0]}) 
\begin{align*}
\E\l[ \alpha^2\beta^0\r]&=\frac{2}{4\bg_n}\l[ t^{n+2}\r]\frac{1}{(1-t)^2}\(\l[u^2\r]\(g_2\(te^{(1-t)u}\)\)^2 \l[v^0 \r]\(g_0\(te^{(1-t)v}\)-1\)^2\)\\
    &=\frac{1}{2\bg_n}\l[ t^{n+2}\r]\frac{2t^4\(t^4+10 t^3 +20 t^2 +10 t +1\)}{(1-t)^{10}} \\
    &= \frac{1}{\bg_n}\Bigg( \binom{n-6+10-1}{n-6}+10\binom{n-5+10-1}{n-5}+20\binom{n-4+10-1}{n-4}\\
    & \qquad\qquad\qquad\qquad\qquad\qquad +10\binom{n-3+10-1}{n-3}+\binom{n-2+10-1}{n-2} \Bigg)\\
    &= \frac{(n+3)(n+2)^2(n+1)^2 n}{1440}.
\end{align*}
We finish the calculation by applying the variance formula $\Var\l[ \alpha \r] = \E \l[ \alpha^2 \r] - \E \l[ \alpha \r]^2$.
\end{proof}

\noindent We also have the following, which will be needed in Section \ref{sec-comparability}.

\begin{theorem}\label{thm-Eab}
We have 
\[
\E\l[\alpha\beta\r]=\frac{\bg_{n+1}\bg_{n+4}\(n^2+4 n+6\)}{420n(n+1)}.
\]
\end{theorem}

\begin{proof}
From (\ref{eq-E[a^1b^0]}) and Theorem \ref{thm-E[a^kb^l]} we have
\begin{align*}
\E\l[ \alpha^1\beta^1\r]&=\frac{1}{4\bg_n}\l[ t^{n+2}\r]\frac{1}{(1-t)^2}\(\l[u^1\r]\(g_1\(te^{(1-t)u}\)\)^2 \)^2
    =\frac{1}{4\bg_n}\l[ t^{n+2}\r]\frac{4t^4(1+t)^2}{(1-t)^{10}}\\
    &= \frac{1}{\bg_n}\(\binom{n-2+10-1}{n-2}+2\binom{n-3+10-1}{n-3}+\binom{n-4+10-1}{n-4}\)\\
    &=\frac{(n+5)(n+4)(n+3)(n+2)\(n^2+4 n+6\)}{15120},
\end{align*}
from which the result follows. 
\end{proof}

We next show that, for the uniformly random $\w\in\bG_n$, with high probability (whp) the $i$th length vector coordinate $\ell_i=\ell_i\(\w\)$ is of order $n$ for each $i\in [4]$. Indeed, recall that by Corollary \ref{cor-BGcompCrit} we have $\beta = \frac{1}{2}\ell_2\ell_3(\ell_2+\ell_3)$.  Furthermore, $\ell_2,\ell_3,\ell_2+\ell_3-1 \in [n-1]$ and $\ell_2$ has density
\begin{equation}\label{eqn-ell2density}
\pr(\ell_2=x) = \frac{1}{\bg_n} \sum_{i=0}^{n-1-x} (n-x-i) = \frac{1}{\bg_n} \binom{n+1-x}{2};
\end{equation}
here, on the event $\{\ell_2=x=b-a\}$ (with $a,b,c$ defined as in (\ref{eq-BGform})) there are a total of $(n-x)$ pairs $(a,b)$ with $0\le a<b<n$, and for each pair $(a,b)$ with $b-a=x$ and $a=i$ we have $(n-x-i)$ possible $c\in (x+i,n]=(b,n]$, from which the sum in (\ref{eqn-ell2density}) follows. Now we obtain the cumulative distribution function:
\begin{equation*}
\pr(\ell_2 \leq s) = \frac{1}{\bg_n} \sum_{x=1}^s \binom{n+1-x}{2} = \frac{1}{\bg_n} \left[ \binom{n+1}{3} - \binom{n+1-s}{3}\right] = 1-\frac{\bg_{n-s}}{\bg_n}.
\end{equation*}
Therefore
\begin{equation}\label{eqn-ell2cdf}
\pr(\ell_2 \leq s) = \frac{s^3-3s^2n+3sn^2-s}{n^3-n}=\frac{n^3-(n-s)^3-s}{n^3-n},\qquad 1\le s \le n-1,
\end{equation}
so it is clear that if $s=o(n)$ then $\pr(\ell_2\leq s)\to 0$, $n\to\infty$. Thus, by equidistribution of the $\ell_i$, it follows that the coordinates of the length vector are whp of order $n$. We summarize this in the following theorem.

\begin{theorem}\label{thm-betaOrder}
For the uniformly random $\w\in\bG_n$, whp we have $\alpha=\Theta(n^3)$ and $\beta=\Theta(n^3)$.
\end{theorem}

\begin{proof}
We need only observe that the $\ell_i$ are equidistributed, and thus by (\ref{eqn-ell2cdf}) each of these random length vector coordinates is of almost sure order $n$. A union bound now finishes the proof. 
\end{proof}

This last theorem highly suggests, but does not prove, that scaling $\beta$ ($\alpha$ respectively) by $n^3$ could lead us to its limiting distribution.
But as $\ell_2$ and $\ell_3$ are equidistributed, analogous formulas (\ref{eqn-ell2density}) and (\ref{eqn-ell2cdf}) hold for $\ell_3$. And for the joint density of $\ell_2$ and $\ell_3$, given $x,y\ge 1$ with $x+y\le n$ we have
\begin{equation}\label{eqn-betajoint}
\pr\(\ell_2=x,\ell_3=y\) = \frac{1}{\bg_n} \sum_{i=0}^{n-x-y} 1 = \frac{n+1-x-y}{\bg_n}=6\(1-\frac{x}{n+1}-\frac{y}{n+1}\)\frac{1}{n(n-1)}.
\end{equation}
This last expression is the generic term of a Riemann sum, and so we recall the following classical distribution. Let $3$ points be chosen independently and uniformly at random from the unit interval $(0,1)$, and let $X<Y<Z$ represent these points in their increasing order. Define $L_1=X$, $L_2=Y-X$, $L_3=Z-Y$ and $L_4=1-Z$ to be the random lengths of consecutive subintervals formed by these 3 random points. Then it is known (see, e.g., Feller \cite[ch. 1]{Feller}) that the random variables $L_1,\ldots,L_4$ are exchangeable and, for $1\le k \le 3$, the joint density of $(L_1,\ldots,L_k)$ is 
\[
f(x_1,\ldots,x_k):=\frac{3!}{(3-k)!}\(1-\sum_{j=1}^k x_j\)^{3-k}.
\]
Thus, as the $L_i$ are exchangeable it follows that $(L_2,L_3)$ has joint density $f(x_1,x_2)=6(1-x_1-x_2)$. This, combined with (\ref{eqn-betajoint}), gives the following theorem.

\begin{theorem-betaDist}
The random variable $\beta/n^3$ converges in distribution to $\frac{1}{2}L_2L_3(L_2+L_3)$. More precisely, for intervals $I=(ni_0,ni_1)\subseteq (0,n)$ and $J=(nj_0,nj_1)\subseteq (0,n)$ we have 
\[
\pr\(\ell_2\in I, \ell_3\in J\)
\to \int_{\substack{i_0<x_1<i_1 \\ j_0<x_2<j_1\\ 0< x_1+x_2 < 1}} 6(1-x_1-x_2) \,dA=\pr\l[L_2\in (i_0,i_1),L_3\in(j_0,j_1)\r],
\]
$n\to\infty$, with error term of order $n^{-1}$.
\end{theorem-betaDist}

\begin{proof}
From (\ref{eqn-betajoint}) we have
\begin{align*}
\pr\(\ell_2 \in I, \ell_3 \in J\)
&= \sum_{\substack{i_0<x/n<i_1 \\ j_0<y/n<j_1 \\ 0 < x/n+y/n < 1}} 6\(1-\frac{x}{n}-\frac{y}{n}\)\frac{1}{n^2}+O\(n^{-1}\)\\
&= \int_{\substack{i_0<x_1<i_1 \\ j_0<x_2<j_1\\ 0< x_1+x_2 < 1}} 6(1-x_1-x_2) \,dA + O\(n^{-1}\),
\end{align*}
from which the result follows.
\end{proof}

Thinking of $a$, $b$, and $c$ in (\ref{eq-BGform}) as random variables, we may similarly express probabilities in terms of coordinate ranges for the random vector $(a,b,c)$.

\begin{theorem}\label{thm-asymp beta pdf}
For $\w\in \bG_n$ chosen uniformly at random, we have 
\[
\pr\(na_0 < a < na_1, nb_0 < b < nb_1, nc_0 < c < nc_1\) = \int_{\substack{a_0< x < a_1\\ b_0 < y < b_1 \\ c_0 < z < c_1 \\ 0< x<y<z < 1}} 6 \,dV + O\(n^{-1}\).\epf
\]
\end{theorem}

\noindent In fact, Theorem \ref{thm-betaDist} follows from Theorem \ref{thm-asymp beta pdf} via $x_1=y-x$, $x_2=z-y$, $x_3=1-z$.

Theorem \ref{thm-asymp beta pdf} readily gives that $\beta$ is not concentrated about $\E\l[\beta\r]\sim n^3/60$. For example, 
an easy calculation gives that an asymptotic proportion of $1/1000$ elements of $\bG_n$ have $c \leq n/10$, and for each of these the value of $\beta$ is at most $(n/10)^3$. 

We also obtain the asymptotic value of the moments of $\beta$, expressed as an integral.  The proof is similar to the above. 

\begin{theorem}\label{thm-betaMomentsIntegral}
For each fixed integer $k \geq 1$, we have
\[
\E\l[ \(\frac{\beta}{n^3}\)^k \r] = \frac{6}{2^k} \int_{0 \leq x<y<z\leq 1} (y-x)^k(z-x)^k(z-y)^k \, dV + O\(\frac{k}{2^k n}\).\epf
\]
\end{theorem}

We end this section with a consequence of Corollary \ref{thm-E[a^k]sim} and Theorem \ref{thm-betaMomentsIntegral}. Multiplying the former asymptotic formula by $2^k/n^{3k}$, the latter equation by $2^k$, and letting $n\to\infty$ we obtain:

\begin{corollary}\label{cor-Selberg}
For each fixed integer $k \geq 1$, we have
\[
\int_0^1\int_0^1\int_0^1|y-x|^k|z-x|^k|z-y|^k\, dx\,dy\,dz = \frac{2}{(k+1)(3k+2)(2k+1)\binom{2k}{k}}.\epf
\]
\end{corollary}

\noindent This integral formula we have obtained through asymptotic moment calculations is a special case of the \emph{Selberg formula} \cite[p. 402, Theorem 8.1.1]{AndrewsAskeyRoy} 
\begin{align*}
S_m(\theta,\varphi,\xi)&=\int_0^1\cdots\int_0^1 \prod_{i=1}^m {x_i}^{\theta-1}(1-x_i)^{\varphi-1}
                                                            \prod_{1\le i < j \le m }|x_i-x_j|^{2\xi}\, dx_1\cdots dx_m\\
&=\prod_{j=0}^{m-1}
\frac{\Gamma(\theta + j\xi)\Gamma(\varphi + j\xi)\Gamma(1 + (j+1)\xi)}{\Gamma(\theta+\varphi + (n+j-1)\xi)\Gamma(1+\xi)}.
\end{align*}
We encourage the interested reader to compute $S_3(1,1,k/2)$ with the Selberg formula, and note that this is equivalent to Corollary \ref{cor-Selberg}.

\section{Comparability}\label{sec-comparability}

We now turn to the question of how often two independent and uniformly random bigrassmannian permutations $\pi$ and $\sigma$ are comparable. Previously, comparability odds have been studied for the Bruhat order and the weak order on the collection of \emph{all} permutations $\frakS_n$ \cite{Hammett,HammettPittel}, the poset of integer partitions of $[n]$ under dominance order \cite{Pittel1,Pittel2}, and the poset of set partitions of $[n]$ ordered by refinement \cite{Pittel3}. In each of these studies, exact evaluation of the probabilities in question were either out of reach, and exchanged for upper and lower bound estimates, or only asymptotic evaluation was possible. In our case, we are able to deliver precise formulas.

Our results can be generalized slightly, so in this direction we begin with the following definition.

\begin{definition}
Given integers $k,\ell\ge 0$, we define a \emph{$(k,\ell)$-star} as a $3$-tuple of sets of BG permutations $(\Pi,\{\rho\},\Sigma)$ of size $|\Pi|=k$, $|\{\rho\}|=1$ and $|\Sigma|=\ell$, respectively, such that $\pi>\rho>\sigma$ for all $\pi\in\Pi$ and $\sigma\in\Sigma$. If, say, $k=0$ then the initial set of permutations $\Pi$ is understood to be empty. We shall refer to the ``middle'' element $\rho$ in a $(k,\ell)$-star as the \emph{hub}, and the elements of $\Pi$ ($\Sigma$, respectively) as \emph{upper pendant} (\emph{lower pendant}, respectively) elements. 
\end{definition}

\noindent The motivation for our $(k,\ell)$-star terminology is clear. Indeed, restricting to the sub-BG poset determined by a $(k,\ell)$-star, we get a Hasse diagram that looks exactly like a star with hub $\rho$ as the ``bottleneck'' element. 

We let $\binom{\bG_n}{k}$ denote the collection of $k$-sets of BG elements. Then a uniformly random $k$-set has the same probability $\binom{\bg_n}{k}^{-1}$ of being selected. Let $q_n(k,\ell)$ denote the probability that the uniformly random $3$-tuple of sets from $\binom{\bG_n}{k}\times\binom{\bG_n}{1}\times\binom{\bG_n}{\ell}$ forms a $(k,\ell)$-star. We can now prove the following.

\begin{theorem}\label{thm-q_n(k,l)}
For each $k\ge 0$, $\ell\ge 0$, and $n\to\infty$ we have 
\begin{align*}
q_n(k,\ell)=\frac{\E\l[\binom{\alpha-1}{k}\binom{\beta-1}{\ell}\r]}{\binom{\bg_n}{k}\binom{\bg_n}{\ell}}
    \sim \frac{2\cdot 3^{k+\ell}}{(k+\ell+1)\binom{3(k+\ell)+2}{3k+1}(2k+1)\binom{2k}{k}(2\ell+1)\binom{2\ell}{\ell}}.
\end{align*}
\end{theorem}

\begin{proof}
Let $q_n\(k,\ell\,|\, \rho\)$ denote the probability $q_n(k,\ell)$ conditioned on the hub element $\rho\in\bG_n$. Then clearly $q_n\(k,\ell\,|\, \rho\)=\binom{\alpha(\rho)-1}{k}\binom{\beta(\rho)-1}{\ell}\binom{\bg_n}{k}^{-1}\binom{\bg_n}{\ell}^{-1}$, and so 
\begin{equation}\label{eq-q_n(k,l)}
q_n(k,\ell)=\E\l[q_n\(k,\ell\,|\, \rho\)\r]=\frac{\E\l[\binom{\alpha-1}{k}\binom{\beta-1}{\ell}\r]}{\binom{\bg_n}{k}\binom{\bg_n}{\ell}}.    
\end{equation}
Given $x\in\mathbb{R}$ and integer $y\ge 0$, we write $(x)_y:=x(x-1)\cdots(x-y+1)$. Then using the well-known identity $(x)_y=\sum_{i=1}^y s(y,i)x^i$, 
where $s(y,i)$ is the $i$th (signed) \emph{Stirling number of the first kind} of order $y$, from (\ref{eq-q_n(k,l)}) we obtain 
\begin{align}\label{eq-q_n(k,l)2}
q_n(k,\ell)&=\frac{\E\l[(\alpha-1)_k(\beta-1)_\ell\r]}{(\bg_n)_k(\bg_n)_\ell}\nonumber\\
&=\frac{1}{(\bg_n)_k(\bg_n)_\ell}\sum_{\substack{0\le i\le k \\ 0\le j \le \ell}}s(k+1,i+1)s(\ell+1,j+1)\E\l[\alpha^i\beta^j\r].    
\end{align}
The dominant term in this sum  is the one containing the factor $\E\l[\alpha^k\beta^\ell\r]$, by Theorem \ref{thm-E[a^kb^l]sim}. After keeping only this term and ignoring all others, and some simplification, we have the asymptotic value of $q_n(k,\ell)$.
\end{proof}

\noindent Note that this theorem gives $q_n(0,0)\sim 1$, $n\to\infty$, as it should! We are ready to establish our main results for comparability.

\begin{theorem-bgPermComp}
Let $\pi,\sigma\in\bG_n$ be selected independently and uniformly at random, and let 
\[
p_{n,2}:=\pr\( \pi \textrm{ and } \sigma \textrm{ are comparable} \) \qquad \textrm{and}\qquad  p_{n,2,\leq}:= \pr\( \pi \le \sigma \). 
\]
Then for $n\ge 2$,
\[
p_{n,2}=\frac{\bg_{n+2}-5}{5\bg_n} \qquad \textrm{and}\qquad  p_{n,2,\leq}=\frac{\bg_{n+2}}{10\bg_n}.
\]
\end{theorem-bgPermComp}

\noindent Before proceeding with the proof, we mention that this result and its extension to 3-element multichains 
(Theorem \ref{thm-3chains}) below stand in stark contrast to the analogous ``probability-of-comparability'' questions addressed in \cite{Hammett,HammettPittel}, where for uniformly random and independent $\pi,\sigma,\tau\in\frakS_n$ the respective probabilities that $\pi\le\sigma$ and $\pi\le\sigma\le\tau$ were shown to be $O\(n^{-2}\)$ and $O\(n^{-6}\)$, and hence these probabilities tend to $0$ as $n\to\infty$.

\begin{proof}
This follows from Theorem \ref{thm-q_n(k,l)} using $k=1$ and $\ell=0$, which gives $$q_n(1,0) = \frac{\E\l[\alpha-1\r]}{\bg_n} = \frac{\frac{\bg_{n+2}}{10} - 1}{\bg_n}.$$
Therefore
\[
p_{n,2,\leq} = q_n(1,0) + \frac{1}{\bg_n} = \frac{\bg_{n+2}}{10\bg_n},
\]
and the formula for $p_{n,2}$ now follows easily.
\end{proof}

In fact, by utilizing a partial fraction decomposition, we can show that these probabilities \emph{decrease} with $n$ to their limiting values.

\begin{corollary}\label{cor-bgPermComp}
For $n\ge 2$, we have
\[
p_{n,2}= \frac{1}{5}+\frac{6}{5 n}+\sum_{k\ge 1}\bigg(\frac{12}{5 n^{2k}}-\frac{18}{5 n^{2k+1}}\bigg)
\quad\textrm{and}\quad
p_{n,2,\leq}= \frac{1}{10}+\frac{3}{5 n}+\sum_{k\ge 2}\frac{6}{5 n^k}.
\]
Thus $p_{n,2}\downarrow \frac{1}{5}$ and $p_{n,2,\leq}\downarrow \frac{1}{10}$ as $n\to\infty$.
\end{corollary}

\begin{proof}
By Theorem \ref{thm-bgPermComp}, for $n\ge 2$ we have
\[
p_{n,2}=\frac{\bg_{n+2}-5}{5\bg_n}
=\frac{1}{5}+\frac{24}{5n}-\frac{3}{n+1}-\frac{3}{5(n-1)}
=\frac{1}{5}+\frac{6}{5 n}+\sum_{k\ge 1}\bigg(\frac{12}{5 n^{2k}}-\frac{18}{5 n^{2k+1}}\bigg).
\]
The derivation for $p_{n,2,\leq}$ is similar.
\end{proof}

\begin{corollary}\label{cor-bg2chains}
Let $c_{n,2,\leq}$ denote the number of $2$-element multichains 
in the poset $(\bG_n,\le)$, i.e. the number of pairs $(\pi,\sigma)\in\bG_n^2$ such that $\pi\le\sigma$. Then
\[
c_{n,2,\leq}=\bg_n^2 p_{n,2,\leq}=\frac{\bg_n\bg_{n+2}}{10}.\epf
\]
\end{corollary}

\begin{definition}
Given an $r$-tuple $(\pi_1,\ldots,\pi_r)\in\bG_n^r$, we say that ``$\pi_1,\ldots,\pi_r$ are comparable'' if and only if there is some re-ordering of the $\pi_i$s that is an $r$-element multichain 
in the BG poset. That is, if and only if there exists $\phi\in\frakS_r$ such that $\pi_{\phi(1)}\le\cdots\le\pi_{\phi(r)}$. 
\end{definition}

\noindent Note that this is consistent with our notion of comparable pairs. The following is a consequence of Theorem \ref{thm-asymp beta pdf}, but we present an alternate constructive argument in the discrete setting. 

\begin{theorem}
Let $r$ be a fixed positive integer. Then the probability that a uniformly random $r$-tuple $(\pi_1,\ldots,\pi_r)$ is comparable is bounded away from $0$ as $n \to \infty$.
\end{theorem}

\begin{proof}
Let $\pi_1$ satisfy $\ell_i \in [n/4,3n/4]$. There are $\Omega(n^3)$ such $\pi_1$, and for each of these create possible $\pi_2$ by removing $j_1,j_4 \in [n/32,n/16]$ from respective coordinates $\ell_1,\ell_4$ and adding (some partition of) $j_1+j_4$ to $\ell_2$ and $\ell_3$.  This produces $\Omega(n^3)$ possible $\pi_2$ for each $\pi_1$.  Iterating, we produce $\Omega(n^{3r})$ distinct $r$-element multichains 
$\pi_1 \leq \pi_2 \leq \cdots \leq \pi_r$.
\end{proof}

As with pairs, we obtain exact results for triples, here by way of the calculation for $\E\l[\alpha\beta\r]$ (Theorem \ref{thm-Eab}). Our methods do not easily generalize to give exact results for $r$-tuples where $r \geq 4$.

\begin{theorem-3chains}
Let $\pi, \sigma, \tau\in \bG_n$ be selected independently and uniformly at random, and let
\[
p_{n,3}:=\pr\( \pi,\sigma,\tau \textrm{ are comparable} \) \qquad \textrm{and}\qquad  p_{n,3,\leq}:= \pr\( \pi \le \sigma \le \tau \). 
\]
Then for $n\ge 2$,
\[
p_{n,3}= \frac{(n^2+4n+6)\bg_{n+3}\bg_{n+6} - 42(n+6)(n+7)\bg_{n+2} + 420(n+6)(n+7)}{70(n+6)(n+7)\bg_n^2}
\]
and
\[
p_{n,3,\leq}=\frac{(n^2+4 n+6)\bg_{n+3}\bg_{n+6}}{420(n+6)(n+7)\bg_n^2}.
\]
\end{theorem-3chains}

\begin{proof}
Here, we have $$q_n(1,1) = \frac{\E\l[(\alpha-1)(\beta-1)\r]}{\bg_n^2} = \frac{\E\l[\alpha\beta - \alpha - \beta + 1\r]}{\bg_n^2},$$
and so
\[
p_{n,3,\leq} = q_n(1,1) + \frac{\E\l[\alpha-1\r] + \E\l[\beta-1\r] + 1}{\bg_n^2} = \frac{\E\l[\alpha\beta\r]}{\bg_n^2}=\frac{(n^2+4 n+6)\bg_{n+3}\bg_{n+6}}{420(n+6)(n+7)\bg_n^2}.
\]
Furthermore,
\begin{align*}
p_{n,3} &= 6\pr(\pi\leq\sigma\leq \tau) - 6 \pr(\pi \leq \sigma, \sigma=\tau) + 6\pr(\pi=\sigma=\tau)\\
&=6\left( \frac{(n^2+4 n+6)\bg_{n+3}\bg_{n+6}}{420(n+6)(n+7)\bg_n^2} - \frac{\bg_{n+2}}{10\bg_n^2}+\frac{1}{\bg_n^2}\right)\\
&=\frac{(n^2+4n+6)\bg_{n+3}\bg_{n+6} - 42(n+6)(n+7)\bg_{n+2} + 420(n+6)(n+7)}{70(n+6)(n+7)\bg_n^2}.\qedhere
\end{align*} 
\end{proof}

Analogous to Corollary \ref{cor-bgPermComp}, we use a partial fraction decomposition to obtain the following, which shows that again these probabilities \emph{decrease} with $n$ to their limiting values.

\begin{corollary}\label{cor-bgPermComp3}
For $n\ge 2$, we have
%
\[
p_{n,3} = \frac{1}{70} + \frac{36}{140n} + \frac{270}{140n^2} + \frac{612}{140n^3} + \sum_{k \geq 2}\left( \frac{28224k-56610}{140n^{2k}} - \frac{2160k-2772}{140n^{2k+1}}\right)
\]
and
\[
p_{n,3,\leq}= \frac{1}{420} + \frac{6}{140n} + \frac{45}{140n^2} +  \sum_{k \geq 2}\left( \frac{648k-1110}{140n^{2k-1}} + \frac{672k-867}{140n^{2k}}\right).
\]
Thus $p_{n,3} \downarrow \frac{1}{70}$ and $p_{n,3,\leq}\downarrow \frac{1}{420}$ as $n\to\infty$. 
\end{corollary}

\begin{proof}
We have
\begin{align*}
p_{n,3,\leq}&=\frac{1}{420}+\frac{117}{35n}+\frac{69}{280(n+1)}-\frac{993}{280(n-1)}+\frac{12}{7n^2}+\frac{3}{70(n+1)^2}+\frac{33}{14(n-1)^2}\\
&= \frac{1}{420} + \frac{6}{140n} + \frac{45}{140n^2} +  \sum_{k \geq 2}\left( \frac{648k-1110}{140n^{2k-1}} + \frac{672k-867}{140n^{2k}}\right).
\end{align*}
Similarly,
\begin{align*}
p_{n,3} &= \frac{1}{70} - \frac{2736}{140 n} - \frac{19863}{140 (n-1)} + \frac{22635}{140 (n+1)} 
+ \frac{28656}{140 n^2} + \frac{6516}{140 (n-1)^2} + \frac{7596}{140 (n+1)^2} \\
&= \frac{1}{70} + \frac{36}{140n} + \frac{270}{140n^2} + \frac{612}{140n^3} - \frac{162}{140n^4} - \frac{1548}{140n^5}\\
&\qquad\qquad\qquad
+ \sum_{k \geq 3}\left( \frac{28224k-56610}{140n^{2k}} - \frac{2160k-2772}{140n^{2k+1}}\right).
\end{align*}
To see that $p_{n,3} \downarrow \frac{1}{70}$, note that for $n \geq 2$ we have $\frac{270}{140n^2} - \frac{162}{140n^4} > 0$ and $\frac{612}{140n^3} - \frac{1548}{140n^5} > 0$.  Also, when $n \geq 2$ and $k \geq 3$, we have $\frac{n(28224k-56610)-(2160k-2772)}{140n^{2k+1}} > 0$.
\end{proof}

\begin{corollary}\label{cor-3chains}
Let $c_{n,3,\leq}$ denote the number of $3$-element multichains 
in the poset $(\bG_n,\leq)$, i.e., the number of triples $(\pi,\sigma,\tau) \in \bG_n^3$ such that $\pi \leq \sigma \leq \tau$.  Then
\[
c_{n,3,\leq}=\bg_n^3 p_{n,3,\leq}=\frac{(n^2+4 n+6)\bg_n\bg_{n+3}\bg_{n+6}}{420(n+6)(n+7)}. \epf
\]
\end{corollary}

For $r$ independent and uniformly random elements $\pi_1,\ldots,\pi_r \in \bG_n$, let $p_{n,r}$ be the probability that they are comparable, and $p_{n,r,\leq}$ the probability that $\pi_1\leq \cdots \leq \pi_r$. It would be interesting to find the exact, or asymptotic, values of $p_{n,r}$ and $p_{n,r,\leq}$ for all $r \geq 4$.

\medskip

\noindent \textbf{Acknowledgements:} We would like to thank Nathan Reading for helpful comments. We also would like to thank the two anonymous referees, whose collective comments much improved the content and readability of the paper. The idea to consider the M\"obius function on intervals of $\bG_n$ came from one of these referees.


\end{document}